\documentclass[10pt,reqno]{amsart}
\usepackage{latexsym,amsmath,amssymb,amscd}
\usepackage[all]{xy}
\usepackage{amsthm}
\usepackage{enumerate}

\makeatletter
  \newcommand\@dotsep{4}
  \def\@tocline#1#2#3#4#5#6#7{\relax
     \ifnum #1>\c@tocdepth 
     \else
     \par \addpenalty\@secpenalty\addvspace{#2}%
     \begingroup \hyphenpenalty\@M
     \@ifempty{#4}{%
     \@tempdima\csname r@tocindent\number#1\endcsname\relax
        }{%
         \@tempdima#4\relax
           }%
      \parindent\z@ \leftskip#3\relax \advance\leftskip\@tempdima\relax
      \rightskip\@pnumwidth plus1em \parfillskip-\@pnumwidth
       #5\leavevmode\hskip-\@tempdima #6\relax
       \leaders\hbox{$\m@th
       \mkern \@dotsep mu\hbox{.}\mkern \@dotsep mu$}\hfill
       \hbox to\@pnumwidth{\@tocpagenum{#7}}\par
       \nobreak
        \endgroup
         \fi}
\makeatother 

\begin{document}
\parskip5pt

\makeatletter
\@addtoreset{figure}{section}
\def\thefigure{\thesection.\@arabic\c@figure}
\def\fps@figure{h,t}
\@addtoreset{table}{bsection}

\def\thetable{\thesection.\@arabic\c@table}
\def\fps@table{h, t}
\@addtoreset{equation}{section}
\def\theequation{
\arabic{equation}}
\makeatother

\newcommand{\bfi}{\bfseries\itshape}

\newtheorem{theorem}{Theorem}
\newtheorem{acknowledgment}[theorem]{Acknowledgment}
\newtheorem{corollary}[theorem]{Corollary}
\newtheorem{lemma}[theorem]{Lemma}
\newtheorem{notation}[theorem]{Notation}

\newtheorem{problem}[theorem]{Problem}
\newtheorem{proposition}[theorem]{Proposition}
\newtheorem{remark}[theorem]{Remark}

\newtheorem{setting}[theorem]{Setting}
\newtheorem{hypothesis}[theorem]{Hypothesis}
\newtheorem{conjecture}[theorem]{Conjecture}
\theoremstyle{definition}
\newtheorem{definition}[theorem]{Definition}
\newtheorem{example}[theorem]{Example}

\numberwithin{theorem}{section}
\numberwithin{equation}{section}

\renewcommand{\1}{{\bf 1}}
\newcommand{\Ad}{{\rm Ad}}
\newcommand{\Alg}{{\rm Alg}\,}
\newcommand{\Aut}{{\rm Aut}\,}
\newcommand{\ad}{{\rm ad}}
\newcommand{\Borel}{{\rm Borel}}
\newcommand{\botimes}{\bar{\otimes}}
\newcommand{\Comm}{{\rm Comm}}
\newcommand{\Cpol}{{\mathcal C}^\infty_{\rm pol}}
\newcommand{\card}{{\rm card}\,}
\newcommand{\Der}{{\rm Der}\,}
\newcommand{\Diff}{{\rm Diff}\,}
\newcommand{\de}{{\rm d}}
\newcommand{\ee}{{\rm e}}
\newcommand{\End}{{\rm End}\,}
\newcommand{\ev}{{\rm ev}}
\newcommand{\hotimes}{\widehat{\otimes}}
\newcommand{\id}{{\rm id}}
\newcommand{\ie}{{\rm i}}
\newcommand{\iotaR}{\iota^{\rm R}}
\newcommand{\GL}{{\rm GL}}
\newcommand{\gl}{{{\mathfrak g}{\mathfrak l}}}
\newcommand{\Hom}{{\rm Hom}\,}
\newcommand{\Img}{{\rm Im}\,}
\newcommand{\Ind}{{\rm Ind}}
\newcommand{\ind}{{\rm ind}\,}
\newcommand{\Ker}{{\rm Ker}\,}
\newcommand{\Lie}{\text{\bf L}}
\newcommand{\Mt}{{{\mathcal M}_{\text t}}}
\newcommand{\m}{\text{\bf m}}
\newcommand{\pr}{{\rm pr}}
\newcommand{\Ran}{{\rm Ran}\,}
\renewcommand{\Re}{{\rm Re}\,}
\newcommand{\so}{\text{so}}
\newcommand{\spa}{{\rm span}\,}
\newcommand{\Tr}{{\rm Tr}\,}
\newcommand{\tw}{\ast_{\rm tw}}
\newcommand{\Op}{{\rm Op}}
\newcommand{\U}{{\rm U}}
\newcommand{\UCb}{{{\mathcal U}{\mathcal C}_b}}
\newcommand{\weak}{\text{weak}}

\newcommand{\CC}{{\mathbb C}}
\newcommand{\HH}{{\mathbb H}}
\newcommand{\RR}{{\mathbb R}}
\newcommand{\TT}{{\mathbb T}}
\newcommand{\NN}{{\mathbb N}}

\newcommand{\Ac}{{\mathcal A}}
\newcommand{\Bc}{{\mathcal B}}
\newcommand{\Cc}{{ C}}
\newcommand{\Dc}{{\mathcal D}}
\newcommand{\Ec}{{\mathcal E}}
\newcommand{\Fc}{{\mathcal F}}
\newcommand{\Hc}{{\mathcal H}}
\newcommand{\Ic}{{\mathcal I}}
\newcommand{\Jc}{{\mathcal J}}
\newcommand{\Kc}{{\mathcal K}}
\newcommand{\Lc}{{\mathcal L}}
\renewcommand{\Mc}{{\mathcal M}}
\newcommand{\Nc}{{\mathcal N}}
\newcommand{\Oc}{{\mathcal O}}
\newcommand{\Pc}{{\mathcal P}}
\newcommand{\Qc}{{\mathcal Q}}
\newcommand{\Sc}{{\mathcal S}}
\newcommand{\Tc}{{\mathcal T}}
\newcommand{\Vc}{{\mathcal V}}
\newcommand{\Uc}{{\mathcal U}}
\newcommand{\Xc}{{\mathcal X}}
\newcommand{\Yc}{{\mathcal Y}}
\newcommand{\Wig}{{\mathcal W}}

\newcommand{\Bg}{{\mathfrak B}}
\newcommand{\Fg}{{\mathfrak F}}
\newcommand{\Gg}{{\mathfrak G}}
\newcommand{\Ig}{{\mathfrak I}}
\newcommand{\Jg}{{\mathfrak J}}
\newcommand{\Lg}{{\mathfrak L}}
\newcommand{\Pg}{{\mathfrak P}}
\newcommand{\Sg}{{\mathfrak S}}
\newcommand{\Xg}{{\mathfrak X}}
\newcommand{\Yg}{{\mathfrak Y}}
\newcommand{\Zg}{{\mathfrak Z}}

\newcommand{\ag}{{\mathfrak a}}
\newcommand{\bg}{{\mathfrak b}}
\newcommand{\dg}{{\mathfrak d}}
\renewcommand{\gg}{{\mathfrak g}}
\newcommand{\hg}{{\mathfrak h}}
\newcommand{\kg}{{\mathfrak k}}
\newcommand{\mg}{{\mathfrak m}}
\newcommand{\n}{{\mathfrak n}}
\newcommand{\og}{{\mathfrak o}}
\newcommand{\pg}{{\mathfrak p}}
\newcommand{\sg}{{\mathfrak s}}
\newcommand{\tg}{{\mathfrak t}}
\newcommand{\ug}{{\mathfrak u}}
\newcommand{\zg}{{\mathfrak z}}

\renewcommand{\H}{{\mathcal H}}

\newcommand{\LC}{{\rm LC}}
\newcommand{\Subquot}{{\rm SQ}}
\newcommand{\Hausd}{{\rm H}}


\def\no#1{\Vert #1\Vert }

\def\wh#1{\widehat{#1}}

\def\res#1{\vert_{ #1}}

\newcommand{\hake}[1]{\langle #1 \rangle }

\newcommand{\scalar}[2]{(#1 \mid#2) }
\newcommand{\dual}[2]{\langle #1, #2\rangle}

\newcommand{\norm}[1]{\Vert #1 \Vert }
\newcommand{\opn}[1]{\operatorname{#1}}

\makeatletter
\title[Fourier transforms of $C^*$-algebras of nilpotent Lie groups]{Fourier transforms of $C^*$-algebras\\ of nilpotent Lie groups}
\author{Ingrid Belti\c t\u a}  
\author{Daniel Belti\c t\u a}
\author{Jean Ludwig}
\address{Institute of Mathematics ``Simion Stoilow'' 
of the Romanian Academy,   
P.O. Box 1-764, Bucharest, Romania}
\email{ingrid.beltita@gmail.com, Ingrid.Beltita@imar.ro}

\address{Institute of Mathematics ``Simion Stoilow'' 
of the Romanian Academy,   
P.O. Box 1-764, Bucharest, Romania}
\email{beltita@gmail.com, Daniel.Beltita@imar.ro}

\address{Universit\'e de Lorraine, Institut Elie Cartan de Lorraine, UMR 7502, Metz, F-57045, France}
\email{jean.ludwig@univ-lorraine.fr}
\thanks{This research has been partially supported by  the Grant
of the Romanian National Authority for Scientific Research, CNCS-UEFISCDI,
project number PN-II-ID-PCE-2011-3-0131. }
\makeatother

\begin{abstract}
For any nilpotent Lie group $G$ we provide a description of the image of its $C^*$-algebra 
through its operator-valued Fourier transform. 
Specifically, 
we show that $C^*(G)$ admits a finite composition series such that 
that the spectra of the corresponding quotients 
are Hausdorff sets in the relative topology, 
defined in terms of the fine stratification of the space of coadjoint orbits of $G$,  and  the canonical fields of 
elementary 
$C^*$-algebras defined by the successive subquotients are trivial. 
We give a description of the image of the Fourier 
transform  as a $C^*$-algebra of piecewise continuous  operator fields on the spectrum,  
determined by the boundary behavior 
of the restrictions of operator fields to the spectra of the subquotients in the composition series. 
For uncountable families of 3-step nilpotent Lie groups 
and also for a sequence of nilpotent Lie groups of arbitrarily high nilpotency step, we prove that every continuous trace subquotient of their $C^*$-algebras has its Dixmier-Douady invariant equal to zero.
\end{abstract}

\maketitle


\section{Introduction}

The classical Fourier transform gives a $*$-isomorphism $C^*(\Vc)\simeq\Cc_0(\Vc^*)$ for any finite-dimensional real vector space $\Vc$ 
regarded as an abelian Lie group. 
Some of the main results of this paper give a result of this type 
when the abelian group $(\Vc,+)$ is replaced by 
an arbitrary connected, simply connected nilpotent Lie group~$G$. 
(See Theorem~\ref{nilpoare} and Corollary~\ref{disb}.)  
If $G$ is non-commutative, then its unitary dual space $\widehat{G}$ is not Hausdorff and moreover $C^*(G)$ is non-commutative, 
hence one must replace the classical Fourier transform by 
a suitable operator-valued Fourier transform which realizes $C^*(G)$ as a $C^*$-algebra of operator fields on $\widehat{G}$. 

The problem of describing the image of the operator-valued Fourier transform on a nilpotent Lie group is notoriously difficult 
and some of its aspects were nicely discussed in \cite{LiRo96}. 

In the present paper we develop a new approach to operator-valued Fourier transforms of $C^*$-algebras of nilpotent Lie groups, 
blending complete positivity techniques and existence of suitable global  canonical 
symplectic coordinates on coadjoint orbits of nilpotent Lie groups.
Instead of the traditional way of constructing Lie group representations as induced representations, 
we use the Lie theoretic method of constructing firstly the representations of the corresponding Lie algebras 
via canonical coordinates on coadjoint orbits \cite{Pe89}, 
and then we integrate the Lie algebra representations to Lie group representations. 
This  leads to a realization of the operator-valued Fourier transform, 
which is studied with the help of some basic properties of $C^*$-algebra extensions and of completely positive maps.  

Specifically, we prove that for every connected, simply connected nilpotent Lie group $G$ its $C^*$-algebra admits 
a finite sequence of closed two-sided ideals 
\begin{equation}\label{introd_eq1}
\{0\}=\Jc_0\subseteq\Jc_1\subseteq\cdots\subseteq\Jc_n=C^*(G)
\end{equation} 
with $*$-isomorphisms $\Jc_j/\Jc_{j-1}\simeq\Cc_0(\Gamma_j,\Kc(\Hc_j))$ 
for suitable locally compact spaces $\Gamma_j$ that are homeomorphic to real semi-algebraic cones, 
where $\Gamma_1$ is a Zariski open subset of $\RR^k$, denoting by $k$ the
codimension of generic coadjoint orbits of $G$. 
Here $\Hc_j$ for $j=1,\dots,n$ are  complex separable Hilbert spaces 
with $\dim\Hc_1=\cdots=\dim\Hc_{n-1}=\infty$,  and $\dim\Hc_n=1$. 
Using a direct sum of topological spaces, one obtains a continuous bijection 
\begin{equation}\label{bijmap}
\Gamma_1\sqcup\cdots\sqcup\Gamma_n\to\widehat{G}.
\end{equation}
whose restriction to $\Gamma_j$ is a homeomorphism onto its image for $j=1,\dots,n$. 
The above map is never a homeomorphism if $G$ is non-abelian, 
since for instance it turns out that $\Gamma_j$ is a relatively dense open subset of $\Gamma_j\cup\Gamma_{j+1}\cup\cdots\cup\Gamma_n$ 
for $j=1,\dots,n-1$ (see Theorem~\ref{nilpsolv} and Definition~\ref{solvspecl_def} below). 
We describe the image of the operator-valued Fourier transform 
of $C^*(G)$ as a $C^*$-algebra of operator fields on $\widehat{G}$, 
determined via the boundary behavior 
of the restrictions of operator fields to $\Gamma_j$ for $j=1,\dots,n$ 
(see Theorem~\ref{nilpoare} and Corollary~\ref{disb} below). 
The intricate topological nature of the $C^*$-algebra extensions involved in this picture can already be seen 
in the case of the Heisenberg group; see \cite{Vo81}. 

For nilpotent Lie groups of step two and dimension $\le 6$ (and a solvable Lie group), 
composition series like those in \eqref{introd_eq1} were constructed in 
\cite[Ch.~6]{Ec96} using techniques of twisted crossed products.
The results thus obtained are optimal in terms  
of the length of the composition sequence, for the case of the free 
two-step nilpotent Lie group with tree generators, see \cite[Ex.~6.3.5]{Ec96}.

In the particular cases of the Heisenberg, threadlike Lie groups, 
and  all nilpotent Lie  groups of dimension $\le 6$, 
explicit calculations have been used to provide descriptions  
of the image of the operator-valued Fourier in   \cite{LuTu11}, \cite{LuRe15} and \cite{ReLu13}.

Our results can also be regarded as a sharpening of the results of \cite{Pe84}, 
where one proved that there exists a composition series as in \eqref{introd_eq1}, 
where the successive quotients are however $C^*$-algebras with continuous trace.

The present paper is structured as follows. 
Section~\ref{section2} contains basic notations and results 
on solvable $C^*$-algebras and norm-continuous operator fields. 
Section~\ref{section3} contains abstract results
on liminary $C^*$-algebras that have a finite composition series 
such that the spectra of the corresponding quotients 
are Hausdorff sets in the relative topology, and  the canonical field of elementary 
$C^*$-algebras defined by the successive subquotients are trivial. 
In particular, we give a description of the image of the Fourier 
transform  for these algebras as  $C^*$-algebras of piecewise continuous  operator fields on the spectrum,  
determined by the boundary behavior 
of the restrictions of operator fields to the spectra of the successive subquotients in the composition series 
(see Theorem~\ref{newcor} below).  
Section~\ref{section4} shows that the results in the previous section can be applied in the case of general connected simply connected nilpotent Lie groups.
 
As an illustration of our abstract results, we show in Section~\ref{section5} that the Heisenberg group 
is uniquely determined in terms of the above structures. 

One of the open problems in this area is the conjecture that every continuous-trace subquotient of $C^*(G)$, where $G$ is an exponential Lie group, has its Dixmier-Douady invariant equal to zero \cite{RaRo88, Ro94, LiRo96}. 
This was proved so far only for 2-step nilpotent Lie groups \cite[Th. 3.4]{LiRo96}; see also the concluding remarks in \cite{Ec96}. 
Using different ideas we will also prove it here for uncountable families of 3-step nilpotent Lie groups 
and also for a sequence of nilpotent Lie groups of arbitrarily high nilpotency step 
(Corollaries \ref{SQ8} and \ref{SQ12}).

\section{Preliminaries}\label{section2}

In this section, beside recalling some  basic notations and results 
on solvable $C^*$-algebras, we record some tools for proving norm-continuity 
of operator fields.

\subsection{Basic notations}\label{basic}
\begin{notation}
\normalfont
We denote by $\Kc(\Hc)$ the $C^*$-algebra of all compact operators 
on some complex separable infinite-dimensional Hilbert space~$\Hc$. 
We denote by $\Sg_p(\Hc)$ the $p$th Schatten ideal, for $1\le p\le\infty$.    
For any $C^*$-algebra $\Ac$ and locally compact Hausdorff space $X$ 
we denote by $\Cc_0(X,\Ac)$ the $C^*$-algebra of continuous $\Ac$-valued functions on~$X$ 
which vanish at infinity. 
We also  denote  by $\widehat{\Ac}$ the dual space of $\Ac$ endowed with its Jacobson topology. 
\qed
\end{notation}

The next lemma is well known, but we record it in order to fix some notation. 

\begin{lemma}\label{L1}
Let $\Ac$ be any $C^*$-algebra. 
For any closed two-sided ideal $\Jc$ of $\Ac$ 
denote 
$$\begin{aligned}
\widehat{\Ac}_{\Jc}&:=\{[\pi]\in\widehat{\Ac}\mid \Jc\subseteq\Ker\pi\}\simeq\widehat{\Ac/\Jc}, \\
\widehat{\Ac}^{\Jc}&:=\{[\pi]\in\widehat{\Ac}\mid \Jc\not\subset\Ker\pi\}\simeq\widehat{\Jc}. 
\end{aligned} $$
Then the following assertions hold: 
\begin{enumerate}
\item $\Jc\mapsto \widehat{\Jc}$ is an increasing bijection between the 
closed two-sided ideals of $\Ac$ and the open subsets of $\widehat{\Ac}$. 
\item $\Jc\mapsto \widehat{\Ac/\Jc}$ is a decreasing bijection between the 
closed two-sided ideals of $\Ac$ and the closed subsets of $\widehat{\Ac}$. 
\end{enumerate}
\end{lemma}

\begin{proof}
See \cite[Props. 2.11.2, 3.2.2]{Dix64}.
\end{proof} 

\begin{remark}\label{S3}
\normalfont 
In the present aper, a key role is played by the well-known fact that for any $C^*$-algebra $\Ac$ 
the locally closed subsets of $\widehat{\Ac}$ 
are precisely the spectra of subquotients of $\Ac$. 

In fact, recall that locally closed subset means any 
set of the form $F\cap D$, where $F$ is any closed subset of $\widehat{\Ac}$ 
while $D$ is any open subset of $\widehat{\Ac}$. 
Moreover, a subquotient of $\Ac$ is any $C^*$-algebra of the form $\Jc_2/\Jc_1$, 
where $\Jc_1\subseteq\Jc_2$ are any closed two-sided ideals of $\Ac$. 
For any such a pair of ideals, it follows by Lemma~\ref{L1} that 
$D:=\widehat{\Jc_1}$ is an open subset of $\widehat{\Ac}$ 
and $\widehat{\Jc_2/\Jc_1}$ is a closed subset of the open set $\widehat{\Jc_1}$, 
hence it is easily checked that the disjoint union $F:=\widehat{\Jc_2/\Jc_1}\cup(\widehat{\Ac}\setminus\widehat{\Jc_1})$ is 
a closed subset of $\widehat{\Ac}$ and $F\cap D=\widehat{\Jc_2/\Jc_1}$ is locally closed. 
Conversely, for any locally closed subset $F\cap D\subseteq \widehat{\Ac}$ we have 
$F\cap D=D\setminus(D\setminus F)$, 
where $D\setminus F=D\cap(\widehat{\Ac}\setminus F)$ is an open subset of $D$. 
Hence by Lemma~\ref{L1} there exist uniquely determined two-sided closed ideals $\Jc_1\subseteq\Jc_2$ 
of $\Ac$ with $\widehat{\Jc_1}=D\setminus F$ and $\widehat{\Jc_2}=D$. 
Moreover, since $\Jc_1$ is in particular an ideal of $\Jc_2$, 
it follows by Lemma~\ref{L1} again that $\widehat{\Jc_2/\Jc_1}=D\setminus(D\setminus F)=D\cap F$. 
See \cite[Lemma~7.3.5]{Ph87} for the fact that the $*$-isomorphism class of the $C^*$-algebra $\Jc_2/\Jc_1$
depends only on the locally closed set $\widehat{\Jc}_2\setminus \widehat{\Jc}_1$. \qed
\end{remark}

\begin{definition}\label{separate}
Let $X$ be a topological space. 

i) A  point $\gamma\in X$ is said to be \textit{separated  in} $X$ 
if for every $\gamma'\in X$ that is not in the closure of the set $\{\gamma\}$, 
there exist open subsets $V$, $V'\subset X$ such that $\gamma\in V$, $\gamma'\in V'$ and $V\cap V' =\emptyset$.
 
ii)  We denote by $\opn{Sep}(X)$ the collection of subsets $Y\subset X$ such that 
all points in $Y$ are closed and separated in $X$.
\end{definition}

\subsection{Special solvable $C^*$-algebras}

We first recall a notion introduced in \cite{Dy78} 
(see also \cite[Sect. 2]{HY88} for several completions). 

\begin{definition}\label{solv_def}
\normalfont
A $C^*$-algebra $\Ac$ is called \emph{solvable} if it has a \emph{solving series}, 
that is, a finite sequence of ideals 
$\{0\}=\Jc_0\subseteq\Jc_1\subseteq\cdots\subseteq\Jc_n=\Ac$ 
with $*$-isomorphisms $\Jc_j/\Jc_{j-1}\simeq\Cc_0(\Gamma_j,\Kc(\Hc_j))$ 
for suitable locally compact spaces $\Gamma_j$ and complex Hilbert spaces $\Hc_j$ for $j=1,\dots,n$, 
with $\dim\Hc_1\ge\cdots\ge\dim\Hc_n$. 
If $n\ge1$ is the least integer for which there exists a sequence of ideals as above, 
then $n-1$ is called the \emph{length} of the $C^*$-algebra $\Ac$.
\end{definition} 

\begin{definition}\label{R-space}
\normalfont
A \emph{topological $\RR$-space} is a topological space $X$ endowed with a 
continuous map $\RR\times X\to X$, $(t,x)\mapsto t\cdot x$, and with a distinguished point $x_0\in X$ 
satisfying the following conditions: 
\begin{enumerate}
\item For every $x\in X$ and $t\in \RR$ one has $0\cdot x= t \cdot x_0$. 
\item For all $t,s\in\RR$ and $x\in X$ one has $t\cdot(s\cdot x)=ts\cdot x$. 
\item For every $x\in X\setminus \{x_0\}$ the map $\RR\to X$, $t\mapsto t\cdot x$ is a homeomorphism onto its image. 
\end{enumerate}
An \emph{$\RR$-subspace} of the topological $\RR$-space $X$ is any subset $\Gamma\subseteq X$ 
such that $\RR\cdot \Gamma\subseteq\Gamma\cup\{x_0\}$. 
If this is the case, then $\Gamma\cup\{x_0\}$ is a topological $\RR$-space on its own. 

In the above framework, a function $\varphi\colon X\to\RR$ is called \emph{homogeneous} if there exists $r\in[0,\infty)$ 
such that $\varphi(t\cdot x)=t^r\varphi(x)$ for all $t\in\RR$ and $x\in X$. 
\end{definition} 

\begin{example}
\normalfont
Every finite-dimensional real vector space is a topological $\RR$-space. 
Moreover, if $\varphi_1,\dots,\varphi_{n_1},\psi_1,\dots,\psi_{n_2}\colon\RR^m\to\RR$ 
are any homogeneous polynomials, then the semi-algebraic cone 
$$\Gamma:=\{x\in\RR^n\mid\varphi_{j_1}(x)=0\ne\psi_{j_2}(x)\text{ for }1\le j_1\le n_1\text{ and }1\le j_2\le n_2\}$$
is an $\RR$-subspace of $\RR^m$ in the sense of Definition~\ref{R-space}. 

As another type of examples, if $G$ is any nilpotent Lie group, then the topological $\RR$-space structure of $\gg^*$ 
gives rise to a topological $\RR$-space structure of the orbit space $\gg^*/G$ (hence also of the dual space $\widehat{G}$ 
via Kirillov's correspondence), and the vector space of characters $[\gg,\gg]^\perp$, 
viewed as the set of singleton orbits, is an $\RR$-subspace of $\gg^*/G$. 
More specifically, the $\RR$-space structure of $\gg^*/G$ is the map 
$$\RR\times (\gg^*/G)\to\gg^*/G, \quad (t,\Oc_{\xi})\mapsto\Oc_{t\xi}$$
where we denote by $\Oc_\xi$ the coadjoint orbit of every $\xi\in\gg^*$. 
\end{example}

\begin{lemma}\label{dom}
Let $X$ be any topological space and for $j=1,2$ let $V_j$ be any open subset of $X$ 
that is homeomorphic to an open subset of $\RR^{r_j}$, where $r_j\ge1$ is some integer. 
If $V_1\cap V_2\ne\emptyset$, then $r_1=r_2$. 
\end{lemma}

\begin{proof}
The nonempty open set $V:=V_1\cap V_2$ is homeomorphic to some open subsets of $\RR^{r_1}$ and of $\RR^{r_2}$, 
hence $r_1=r_2$ by Brouwer's theorem on the invariance of domain. 
\end{proof}

\begin{definition}\label{solvspecl_def}
\normalfont
We say that $\Ac$ is a \emph{special solvable $C^*$-algebra} if 
it is separable and 
it has a \emph{special solving series}, 
that is, a solving series as in Definition~\ref{solv_def} with the following additional properties: 
\begin{enumerate}
\item\label{solvspecl_def_item1} 
$\widehat{\Ac}$ has the structure of a topological $\RR$-space 
and $\Gamma_j\subseteq\widehat{\Ac}$ is an $\RR$-subspace for $j=1,\dots,n$. 
\item\label{solvspecl_def_item2} 
One has $\dim\Hc_n=1$ and $\Gamma_n$ is isomorphic as a topological $\RR$-space to a finite-dimensional vector space. 
\item\label{solvspecl_def_item3} 
For $j=1,\dots,n-1$ one has $\dim\Hc_j=\infty$, the set $\Gamma_{j+1}$ is dense in $\widehat{\Ac}\setminus\widehat{\Jc}_{j}$, 
and the points of $\Gamma_{j+1}$ are closed and separated in $\widehat{\Ac}\setminus\widehat{\Jc}_{j}$.  
\item For $j=1,\dots,n$, $\Gamma_j$ is isomorphic as a topological $\RR$-space to a semi-algebraic cone 
in a finite-dimensional vector space. 
In addition, $\Gamma_1$ is assumed to be a Zariski open set,  
and the dimension of the corresponding ambient vector space is called the \emph{index of $\Ac$} and is denoted by $\ind\Ac$. 
\item For $j=1,\dots,n$, there exists a homogeneous function $\varphi_j\colon\widehat{\Ac}\to\RR$  
such that $\varphi_j\vert_{\Gamma_1}$ is a polynomial function (via the above homeomorphism) 
and 
$$\Gamma_j=\{\gamma\in\widehat{\Ac}\mid \varphi_j(\gamma)\ne0\text{ and }\varphi_i(\gamma)=0\text{ if }i<j\}.$$ 
\end{enumerate}
\end{definition}

Since $\Gamma_1$ is open and dense in $\widehat{\Ac}$, 
it follows by Lemma~\ref{dom} that $\ind\Ac$ does not depend on the choice of the solving series of~$\Ac$. 

\begin{remark}
\normalfont
In Definition~\ref{solvspecl_def}, the distinguished point of the topological $\RR$-space $\widehat{\Ac}$ 
is the origin $0\in\Gamma_n$ of the vector space~$\Gamma_n$. 
Therefore, for $j=1,\dots,n$ and every $\gamma\in\Gamma_j$, one has 
$\overline{\RR^*\cdot \gamma}\cap \RR \cdot \gamma =\{0\}\subset \Gamma_n$. 
\qed
\end{remark}

\begin{remark}
\normalfont
In Definition~\ref{solvspecl_def}, 
since $\dim\Hc_n=1<\dim\Hc_j$ for $1\le j<  n$, 
the set $\Gamma_n$ is precisely the set of characters of $\Ac$, 
that is, the non-zero $*$-homomorphisms $\chi\colon\Ac\to\CC$. 
Denoting by $\Comm(\Ac)$ the closed two-sided ideal of $\Ac$ generated by the set $\{ab-ba\mid a,b\in\Ac\}$, 
it follows that $\Ac/\Comm(\Ac)$ is a commutative $C^*$-algebra and for every $\chi\in\Gamma_n$ 
one has $\Comm(\Ac)\subseteq\Ker\chi$, hence $\chi$ can be identified with an element in the spectrum of  
$\Ac/\Comm(\Ac)$. 
Thus the Gelfand representation provides a $*$-isomorphism 
$\Ac/\Comm(\Ac)\simeq \Cc(\Gamma_n)$, 
and the condition in Definition~\ref{solvspecl_def}\eqref{solvspecl_def_item2} 
implies that the spectrum of the commutative $C^*$-algebra $\Ac/\Comm(\Ac)$ 
is homeomorphic to a finite-dimensional real vector space, 
whose dimension is uniquely determined because of 
Brouwer's theorem on the invariance of domain. 
\qed
\end{remark}

\subsection{On the continuity of operator fields}

The following two lemmas go back to \cite[Prop. 2.2, Th. 2.3]{ReLu13}. 

\begin{lemma}\label{F1}
Let $\Ac$ be any $C^*$-algebra with some subset of its spectrum $\Gamma\subseteq\widehat{\Ac}$ 
such that the relative topology of $\Gamma$ is Hausdorff. 
Assume $\Hc$ is a complex Hilbert space and in every $\gamma\in\Gamma$ 
we have picked $\pi_\gamma\colon\Ac\to\Bc(\Hc)$. 
Also let $\Vc_1$ and $\Vc_2$ be any total subsets of $\Hc$. 

If $a\in\Ac$ has the property that 
for every $v_1\in\Vc_1$ and $v_2\in\Vc_2$, 
the function $\Gamma\to\CC$, $\gamma\mapsto\langle\pi_\gamma(a)v_1,v_2\rangle$, 
is continuous,  
then for every $R\in\Sg_1(\Hc)$ the function 
$$f_R\colon\Gamma\to\CC,\quad f_R(\gamma)=\Tr(\pi_\gamma(a)R)$$
is continuous and bounded. 
\end{lemma}

\begin{proof}
For every $R\in\Sg_1(\Hc)$ and $\gamma\in\Gamma$ we have 
$\vert f_R(\gamma)\vert\le \Vert R\Vert_1\Vert a\Vert$, 
hence $f_R\in\ell^\infty(\Gamma)$ and $\Vert f_R\Vert_\infty\le \Vert R\Vert_1\Vert a\Vert$. 
Since the limit of any uniformly convergent sequence of continuous functions is in turn continuous, 
it then easily follows that the set 
$\{R\in\Sg_1(\Hc)\mid f_R\in\Cc(\Gamma)\}$ is a closed linear subspace of 
$\Sg_1(\Hc)$. 
As that closed linear subspace has a dense linear subspace consisting of rank-one operators by the hypothesis on~$a$, 
the conclusion follows. 
\end{proof}

\begin{lemma}\label{F2}
In the setting of Lemma~\ref{F1}, 
assume $\Sc$ is a dense $*$-subalgebra of $\Ac$ 
for which every element $a\in\Sc$ satisfies the condition of Lemma~\ref{F1}, 
and moreover for all $\gamma\in\Gamma$ we have $\pi_\gamma(a)\in\Sg_1(\Hc)$ and the function 
$\Gamma\to\CC$, $\gamma\mapsto\Tr\pi_\gamma(a)$ is continuous.   

Then for every $a\in\Ac$ the map 
$\Pi_a\colon \Gamma\to\Bc(\Hc)$, $\gamma\mapsto\pi_\gamma(a)$, is norm continuous. 
\end{lemma}

\begin{proof}
For arbitrary $a\in\Ac$ and $\gamma\in\Gamma$ we have 
$\Vert\Pi_a(\gamma)\Vert=\Vert\pi_\gamma(a)\Vert\le\Vert a\Vert$,  
hence just as in the proof of Lemma~\ref{F1} we can see that 
the set $\{a\in\Ac\mid\Pi_a\text{ is continuous}\}$ 
is a closed linear subspace of~$\Ac$. 
Therefore it suffices to check that $\Pi_a$ is continuous for $a\in\Sc$. 

If $a\in\Sc$ and $\{\gamma_j\}_{j\in J}$ is any net in $\Gamma$ 
which is convergent to some $\gamma\in\Gamma$, then for all $j\in J$ we have 
$$\begin{aligned}
\Vert\Pi_a(\gamma_j)-\Pi_a(\gamma)\Vert^2
\le 
&\Vert\Pi_a(\gamma_j)-\Pi_a(\gamma)\Vert_2^2 \\
=&\Vert\pi_{\gamma_j}(a)-\pi_{\gamma}(a)\Vert_2^2 \\
=&\Tr\pi_{\gamma_j}(a^*a)-2\Re\Tr(\pi_{\gamma_j}(a)\pi_{\gamma}(a)^*)+\Tr\pi_{\gamma}(a^*a).
\end{aligned}$$
Since $a^*a\in\Sc$, it follows by hypothesis that 
$$\lim\limits_{j\in J}\Tr\pi_{\gamma_j}(a^*a)=\Tr\pi_{\gamma}(a^*a)$$ 
and on the other hand Lemma~\ref{F1} implies 
$$\lim\limits_{j\in J}\Tr(\pi_{\gamma_j}(a)\pi_{\gamma}(a)^*)
=\Tr(\pi_{\gamma}(a)\pi_{\gamma}(a)^*)=\Tr\pi_{\gamma}(a^*a)$$
hence by the above estimate we obtain 
$\lim\limits_{\gamma\in\Gamma}\pi_{\gamma_j}(a)=\pi_\gamma(a)$ in $\Bc(\Hc)$, 
which concludes the proof. 
\end{proof}

\begin{lemma}\label{S4}
Let $\Ac$ be any 
liminary separable 
$C^*$-algebra 
and consider any locally closed subset $\Gamma\subseteq\widehat{\Ac}$ 
for which the relative topology of $\Gamma$ is Hausdorff, 
and let $\Jc_1\subseteq\Jc_2$ be two-sided closed ideals of $\Ac$ 
with $\widehat{\Jc_2/\Jc_1}=\Gamma$ (see Remark~\ref{S3}).
Then the following properties are equivalent:   
\begin{enumerate}
\item\label{S4_item1} 
The canonical field of elementary $C^*$-algebras defined by $\Jc_2/\Jc_1$ 
on~$\Gamma$ is trivial.  
\item\label{S4_item2}
There exist a complex Hilbert space $\Hc$ and a complete system of distinct representatives 
$\{\pi_\gamma\colon\Ac\to\Bc(\Hc)\}_{\gamma\in\Gamma}$ 
of the equivalence classes of representations corresponding to the elements of $\Gamma$ 
such that for every $a\in\Jc_2$ the mapping $\Gamma\to\Bc(\Hc)$, $\gamma\mapsto\pi_\gamma(a)$ 
is continuous with respect to the norm operator topology of $\Bc(\Hc)$. 
\end{enumerate} 
\end{lemma}

\begin{proof}
The implication \eqref{S4_item1}$\Rightarrow$\eqref{S4_item2} is clear, 
so we are left to proving only the converse implication. 

By Lemma~\ref{L1}, we may replace  $\Ac$ by its subquotient $\Jc_2/\Jc_1$.
Thus  
we may assume $\Jc_1=\{0\}$ and $\Jc_2=\Ac$, hence $\Gamma=\widehat{\Ac}$. 
Then the hypothesis~\eqref{S4_item2} shows that the continuous sections of 
the canonical field of elementary $C^*$-algebras defined by $\Ac$ 
(which is a continuous field of $C^*$-algebras by \cite[10.5.1]{Dix64} 
since $\Gamma=\widehat{\Ac}$ is Hausdorff)
are also continuous sections of the trivial field with the fiber $\Kc(\Hc)$ over $\widehat{\Ac}$. 
Now \cite[Prop. 10.2.4]{Dix64} ensures that the two aforementioned continuous fields of $C^*$-algebras 
are isomorphic, and in particular the canonical field of elementary $C^*$-algebras defined by $\Ac$ 
is trivial. 
\end{proof}

\section{Norm controls for the boundary values of $C^*$-Fourier transforms}\label{section3}

The present section proves our main abstract results on the image of the operator-valued Fourier transform 
of liminary $C^*$-algebras that have a finite composition series 
such that the spectra of the corresponding successive quotients 
are Hausdorff sets in the relative topology, and  the canonical fields of elenentary 
$C^*$-algebras defined by the subquotients are trivial.

\subsection{Boundary values of Fourier transforms}\label{subsect2.1}

\begin{proposition}\label{S1}
Assume the following: 
\begin{itemize}
\item $\Ac$ is any separable nuclear $C^*$-algebra.
\item $T\in \opn{Sep}(\widehat{\Ac})$  is an open dense subset of $\widehat{\Ac}$, 
such that its corresponding ideal of $\Ac$ is $*$-isomorphic to $\Cc_0(T,\Kc)$.
 Denote  by $\Fc_T\colon\Ac\to\Cc_b(T,\Kc)$ the Fourier transform of $\Ac$ restricted to $T$.
\item $0\to\Cc_0(T,\Kc)\hookrightarrow \Ac\mathop{\longrightarrow}\limits^q\Bc\to 0$ 
is an exact sequence of $C^*$-algebras. 
\end{itemize}
Then the following assertions hold: 
\begin{enumerate}
\item The map 
$$\Phi \colon\Ac\to \Cc_b(T,\Kc)\oplus \Bc,\quad \Phi (a)=\Fc_T(a)\oplus q(a)$$ 
is an isometric $*$-homomorphism. 
\item There exists a linear map $\nu\colon\Bc\to\Cc_b(T,\Kc)$ 
which is completely positive, completely isometric, almost $*$-homomorphism, 
and 
$$\Ran\Phi=\Bigl\{f\oplus b \in\Cc_b(T,\Kc)\oplus \Bc \mid \lim\limits_{t\to\infty}f(t)-(\nu(b))(t)=0\Bigr\}.$$
\item There exists a completely isometric cross section of $q$.  
\end{enumerate}
\end{proposition}

\begin{proof}
Since $T$ is the spectrum of the ideal $\Jc=\Cc_0(T,\Kc)$ of $\Ac$,  and the points of $T$ are closed and separated in $\widehat{\Ac}$ sense of \cite[Def. II.6]{De72}, 
it follows by \cite[Prop.~IV.1.3]{De72}  that the Busby invariant $\gamma$ of the extension of $C^*$-algebras 
$$ 0\to\Jc\hookrightarrow \Ac\mathop{\longrightarrow}\limits^q\Bc\to 0 $$
has the range contained into $L(\Jc)/\Jc$, 
where $L(\Jc)$ is the largest two-sided closed ideal of the multiplier algebra of $\Jc$ 
which contains $\Jc$ and for which $\widehat{\Jc}\in\opn{Sep}(\widehat{L(\Jc)})$, that is, the points of $T=\widehat{\Jc}$ 
are closed and separated in $\widehat{L(\Jc)}$ 
(\cite[Lemme III.2.4]{De72}). 
On the other hand, since $\Jc=\Cc_0(T,\Kc)$, 
we have $L(\Jc)=\Cc_b(T,\Kc)$ by \cite[Prop. III.4.1]{De72}. 
Consequently the Busby invariant of the extension of $C^*$-algebras from the statement 
is a $*$-homomorphism 
$$\beta\colon\Bc\to\Cc_b(T,\Kc)/\Cc_0(T,\Kc).$$
Moreover, since $T$ was assumed to be dense in $\widehat{\Ac}$, 
it follows by \cite[Cor. 6.4]{Bu68} that the $*$-homomorphism $\beta$ is injective.  
Therefore $\beta$ is a complete isometry and it is moreover completely positive. 

On the other hand, since the $C^*$-algebra $\Ac$ is nuclear, 
it follows by \cite[Cor. 4]{CE77} that also its quotient $\Bc$ is nuclear. 
Also $\Bc$ is separable since $\Ac$ is. 
It then follows by the completely positive lifting theorem \cite[Th. 3.10]{CE76} 
that there exists a completely positive contraction $\nu\colon\Bc\to\Cc_b(T,\Kc)$ 
with the property $\beta(b)=\nu(b)+\Cc_0(T,\Kc)$ for all $b\in\Bc$. 
The unital extension of $\nu$ to the unitizations of the $C^*$-algebras $\Bc$ and $\Cc_b(T,\Kc)$ 
is completely positive by \cite[Lemma 3.9]{CE76}, 
hence completely contractive by \cite[1.3.3]{BL04}. 
Therefore $\nu$ is both completely positive and completely contractive. 

Now note that the canonical map 
$$\widetilde{q}\colon \Cc_b(T,\Kc)\to\Cc_b(T,\Kc)/\Cc_0(T,\Kc), \quad f\mapsto f+\Cc_0(T,\Kc),$$ 
is a $*$-homomorphism, hence completely contractive, and moreover 
$\widetilde{q}\circ\nu=\beta$ is completely isometric. 
Since we have proved above that $\nu$ is completely contractive, 
it then  follows that $\nu$ is actually completely isometric. 
Moreover, since $\widetilde{q}\circ\nu=\beta$ is a $*$-homomorphism, 
it follows that $\nu$ is an almost $*$-homomorphism,  
which means that for all $b_1,b_2\in\Bc$ we have  
$\lim\limits_{t\to\infty}(\nu(b_1b_2)-\nu(b_1)\nu(b_2))(t)=0$ 
and $\nu(b_1^*)=\nu(b_1)^*$.

For proving the assertion on $\Ran\Phi$, 
note that we have the commutative diagram 
$$\xymatrix{
0 \ar[r] & \Jc \ar[d]^{\Fc\vert_{\Jc}}\ar[r] & \Ac \ar[r]^{q} & \Bc \ar[dl]_{\nu} \ar[d]^{\beta} \ar[r] & 0 \\
0 \ar[r] & \Cc_0(T,\Kc) \ar[r]& \Cc_b(T,\Kc) \ar[r]^{\widetilde{q}\ \ \ } & \Cc_b(T,\Kc)/\Cc_0(T,\Kc) \ar[r] & 0
}$$
and the fact that $\beta$ is the Busby invariant of the extension from the upper row of the above diagram 
implies by  \cite[Proof of Prop. 4.2, Th. 4.3]{Bu68} that 
we have the $*$-isomorphism of $C^*$-algebras 
\begin{equation}\label{S1_proof_eq1}
\Ac\simeq\{f\oplus b \in\Cc_b(T,\Kc)\oplus\Bc\mid\widetilde{q}(f)=\beta(b)\},\quad 
a\mapsto \Fc_T(a)\oplus q(a). 
\end{equation}
Note that for all $f\oplus b\in\Cc_b(T,\Kc)\oplus \Bc$ 
we have 
$$\widetilde{q}(f)=\beta(b)\iff f+\Cc_0(T,\Kc)=\nu(b)+\Cc_0(T,\Kc)
\iff f-\nu(b)\in\Cc_0(T,\Kc)$$
and for every $f\in\Cc_b(T,\Kc)$ there exists at most one $b\in\Bc$ satisfying the above condition, 
since if we have $f-\nu(b_1),f-\nu(b_2)\in\Cc_0(T,\Kc)$, 
then 
$$\beta(b_1)=\nu(b_1)+\Cc_0(T,\Kc)=\nu(b_2)+\Cc_0(T,\Kc)=\beta(b_2)$$ 
hence $b_1=b_2$. 

Finally, since $\nu$ is completely isometric, it follows that the map 
$$\iota \colon \Bc\to \Cc_b(T,\Kc)\oplus \Bc, \quad b\mapsto \nu(b)\oplus b$$
is completely isometric and its image is contained in the image 
of~$\Ac$ by the $*$-isomorphism~\eqref{S1_proof_eq1}. 
By composing the inverse of that $*$-isomorphism with $\iota$ 
we obtain a completely isometric cross section of~$q\colon\Ac\to\Bc$, 
and this completes the proof. 
\end{proof}

\begin{remark}\label{S2}
\normalfont
If we have a short exact sequence $0\to\Jc\to\Ac\to\Bc\to 0$ 
where $\Ac$ is any $C^*$-algebra of type~I, 
then it follows that also its quotient $\Bc$ is of type~I, 
and both $\Ac$ and $\Bc$ are nuclear by \cite[Prop.~1.6, Ch. XV]{Ta03}. 
\qed
\end{remark}

\begin{definition}\label{fourtrans}
Let $\Ac$ be a $C^*$-algebra with spectrum $\widehat{\Ac}$. 
We choose for every $\gamma\in \widehat{\Ac}$ a representation 
$(\pi_\gamma,\H_\gamma)$ in the equivalence class $\gamma$.
Let $\ell^{\infty}(\widehat{\Ac})$ be the algebra of all bounded operator fields defined over $\widehat{\Ac}$ by
$$ \ell^{\infty}(\widehat{\Ac}):=\left\{\phi=(\phi(\pi_\gamma)\in\Bc(\Hc_{\gamma}))_{\gamma\in\widehat{\Ac}}\mid
\norm{\phi}_\infty
:=\sup_{\gamma}\norm{\phi(\pi_\gamma)}_{\Bc(\Hc_\gamma)}<\infty\right\}.$$
We define for $a\in \Ac$ its Fourier transform $ \Fc_{\Ac} (a)=\hat a $ by
$$
\Fc_{\Ac}(a)(\gamma)=\hat{a}(\gamma)=\pi_\gamma(a),\quad  \gamma\in\hat{\Ac}.
$$
Then $\Fc_{\Ac}(a) $ is a bounded field of operators over $\widehat{\Ac}$, and 
the mapping
$$
 \Fc_{\Ac} \colon \Ac \to \ell^\infty(\widehat{\Ac}), \quad  a \mapsto \Fc_{\Ac}(a)
$$
is an isometric $*$-homomorphism. 
 \end{definition}

\begin{theorem}\label{S5}
Let $\Ac$ be any separable liminary $C^*$-algebra 
with an increasing sequence of open subsets of its spectrum 
$$\emptyset=V_0\subseteq V_1\subseteq\cdots\subseteq V_n=\widehat{\Ac} $$
with the corresponding sequence of closed two-sided ideals of $\Ac$ 
$$\{0\}=\Jc_0\subseteq \Jc_1\subseteq\cdots\subseteq\Jc_n=\Ac$$
with $\widehat{\Jc_\ell}=V_\ell$ 
satisfying the following conditions for $\ell=1,\dots,n$: 
\begin{enumerate}
\item\label{S5_item1} 
The set $\Gamma_\ell:=V_\ell\setminus V_{\ell-1}$ is dense in  $\widehat{\Ac}\setminus V_{\ell-1}$.
\item\label{S5_item2} 
There exist a complex Hilbert space $\Hc_\ell$ and a complete system of distinct representatives 
$\{\pi_\gamma\colon\Ac\to\Bc(\Hc_\ell)\}_{\gamma\in \Gamma_\ell}$ 
of the equivalence classes of representations corresponding to the elements of $\Gamma_\ell$ 
such that for every $a\in\Ac$ the mapping $\Gamma_\ell\to\Kc(\Hc_\ell)$, $\gamma\mapsto\pi_\gamma(a)$ 
is continuous with respect to the norm operator topology of $\Kc(\Hc_\ell)$. 
\end{enumerate}
For $\ell=0,\dots,n$  we define 
$$\Lc_\ell:=\{f\colon\widehat{\Ac}\setminus V_\ell\to\Kc(\bigoplus_{j=\ell+1}^n\Hc_j)\mid(\forall j\in\{\ell+1,\dots,n\})
(\forall\gamma\in\Gamma_j)\quad 
f(\gamma)\in\Kc(\Hc_j)\}$$
and 
$$\Fc_{\Ac/\Jc_\ell}\colon\Ac/\Jc_\ell\to\Lc_\ell,\quad (\Fc_{\Ac/\Jc_\ell}(a+\Jc_\ell))(\gamma):=\pi_\gamma(a).$$
Then there is a family of linear maps 
$\{\nu_\ell\colon\Fc_{\Ac/\Jc_\ell} (\Ac/\Jc_\ell) \to\Cc_b(\Gamma_\ell,\Kc(\Hc_\ell))\}_{1\le \ell\le n}$,  
which are completely positive, completely isometric, almost $*$-homomorphisms, 
and for which  the image $\Fc_{\Ac}(\Ac)$ is precisely the set of all $f\in \Lc_0$ such that 
\begin{equation}\label{S5_eq1}
f\vert_{\Gamma_{\ell}} - \nu_\ell(f \vert_{\hat{\Ac}\setminus{V_{\ell}}}) \in \Cc_0(\Gamma_\ell,\Kc(\Hc_\ell))
\end{equation}
for all $\ell\in\{1,\dots,n-1\}$. 
\end{theorem}

\begin{remark}\label{fouriers}
\normalfont 
Under the assumptions in Theorem~\ref{S5},   
for all $a\in \Ac$ we have
$$(\Fc_{\Ac/\Jc_{\ell-1}} (a+\Jc_{\ell-1}))(\gamma)=\pi_\gamma(a)=\Fc_{\Ac}(a)(\gamma) \in\Kc(\Hc_\ell)$$
for all $\gamma\in \hat \Ac\setminus V_{l-1}$. 
This follows from 
\cite[Prop. 2.10.4]{Dix64} and the fact that
$$\hat \Ac \setminus V_{\ell-1}
=\widehat{\Ac/\Jc_{\ell-1}}
=\{[\pi]\in\widehat{\Ac}\mid\Jc_{\ell-1}\subseteq\Ker\pi \}. 
$$
\qed
\end{remark}

\begin{proof}[Proof of Theorem~\ref{S5}]
Recall from Lemma~\ref{L1} and Remark~\ref{S3} that $\widehat{\Jc_\ell/\Jc_{\ell-1}}=V_\ell\setminus V_{\ell-1}$,  
hence Lemma~\ref{S4} implies that 
the canonical field of elementary $C^*$-algebras defined by $\Jc_\ell/\Jc_{\ell-1}$ is trivial. 
Thus we obtain an isometric $*$-isomorphism 
$\Jc_\ell/\Jc_{\ell-1}\simeq\Cc_0(V_\ell\setminus V_{\ell-1},\Kc(\Hc_\ell))$
by the Fourier transform. 
Then the canonical short exact sequence 
$$0\to\Jc_\ell/\Jc_{\ell-1}\to\Ac/\Jc_{\ell-1}\to\Ac/\Jc_\ell\to0$$
leads to a short exact sequence 
$$0\to\Cc_0(V_\ell\setminus V_{\ell-1},\Kc(\Hc_\ell))\to \Ran \Fc_\Ac/\Jc_{\ell-1} \to 
\Ran \Fc_{\Ac/\Jc_{\ell}}\to 0.$$
Here the second arrow is simply the inclusion, while the third is the restriction to $\hat \Ac\setminus V_{\ell}$.

Note that the set $\Gamma_\ell:=V_\ell\setminus V_{\ell-1}$ is in $\opn{Sep}( \widehat{\Ac}\setminus V_{\ell-1})$. 
Indeed, all points of $\widehat{\Ac}$ are closed since $\Ac$ is liminary.
In addition, the function
$[\pi]\mapsto\Vert\pi(a)\Vert$ is continuous
on $\Gamma_\ell$ for every $a\in\Ac$ by \eqref{S5_item2}.
This implies that every point of the open set
$\Gamma_\ell$ separated in
$\widehat{\Ac}$ (see \cite[p. 116, (iii)]{Dix61} or the proof of 
\cite[Th. 2.1]{Fe60}).  
Then by Proposition~\ref{S1}   we obtain 
a completely positive, completely contractive, almost $*$-homomorphism  
$\nu_\ell\colon  \Ran \Fc_{\Ac/\Jc_\ell} \to\Cc_b(\Gamma_\ell,\Kc(\Hc_\ell))$ 
with 
$$
\begin{aligned} \Ran & \Fc_{\Ac/\Jc_{\ell-1}}  =\\
& \Bigl\{f\oplus b\in\Cc_b(\Gamma_\ell,\Kc(\Hc_\ell))\oplus \Ran  \Fc_{\Ac/\Jc_{\ell}} \mid 
f-\nu_\ell(b)\in\Cc_0(\Gamma_\ell,\Kc(\Hc_\ell))\Bigr\}.
\end{aligned} 
$$
We have that  
$$ f\oplus b \vert_{\hat{\Ac} \setminus V_\ell}= b, \quad f\oplus b  \vert_{\Gamma_\ell} = f,$$
hence
when $a\in \Ac$
$$ \Fc_{\Ac/\Jc_{\ell-1}}(a + \Jc_{\ell-1})  \vert_{\Gamma_\ell} - 
\nu_l \big(\Fc_{\Ac/\Jc_{\ell-1}}(a + \Jc_{\ell-1}) \vert_{\hat{\Ac} \setminus V_\ell}\big)\in\Cc_0(\Gamma_\ell,\Kc(\Hc_\ell)).$$
On the other hand, 
for all $\gamma\in \Gamma_\ell$ 
and $a\in\Ac$ we have 
$$(\Fc_{\Ac/\Jc_{\ell-1}} (a+\Jc_{\ell-1}))(\gamma)=\pi_\gamma(a)=\Fc_{\Ac}(a)(\gamma) \in\Kc(\Hc_\ell).$$
Using  now the above description of $\Ran \Fc_{\Ac/\Jc_{\ell-1}}$, 
we see that $\Fc_{\Ac}(a)$ satisfies~\eqref{S5_eq1}. 

Then, by successively considering the cases $\ell=n,n-1,\dots,1$, we obtain  
the description of $\Ran\Fc_{\Ac/\Jc_{\ell-1}}$, 
that for $\ell=1$ concludes the proof. 
\end{proof}

\subsection{$C^*$-algebras with norm controlled dual limits}

\begin{definition}[see also \cite{LuRe15}]\label{lcd} 
i) Let $S$ be  a topological space. 
We say that $S $ is \textit{locally compact of step $\leq d$} if  
there exists a 
finite increasing family $ \emptyset \ne  S_d\subset S_{d-1}\subset\cdots\subset S_0=S $ of closed 
subsets of $S $,  such that   the subsets $\Gamma_d=S_d$ and  
$ \Gamma_i:=S_{i-1}\setminus S_{i}$,  $i=1,\dots, d$,   
are  locally compact and Hausdorff in  their relative topologies.

ii) Let $S$ be locally compact of step $\le d$, and let $\{\Hc_i\}_{i=1, \dots, d}$ be Hilbert spaces. 
For a closed subset $ M \subset S $, denote by $ CB(M) $ the unital $C^*$-algebra of all uniformly 
bounded operator fields  $ (\psi(\gamma)\in \Bc(\Hc_i))_{\gamma\in S\cap \Gamma_i, i=1,\dots, d}$, 
which are 
operator norm continuous on the subsets $ \Gamma_i \cap M$ for every $  i\in\{1,\dots, d\} $ with  $ \Gamma_i\cap M\ne\emptyset $. 
We provide the algebra $ CB(M) $ with the infinity-norm
$$
\norm{\varphi}_{M}=\sup\left\{\norm{\varphi(\gamma)}_{\Bc(\Hc_i)}\mid M \cap \Gamma_i\ne \emptyset, \, \gamma\in M \cap \Gamma_i\right\}.
$$
\end{definition}

\begin{definition}[see also \cite{LuRe15}]\label{norcontspec}
Let $\Ac $ be a separable liminary $ C^* $-algebra. 
  We assume that the spectrum $\widehat{\Ac} $ of $\Ac$ is a locally compact space of step $\leq d$, 
 $$ \emptyset= S_{d+1}\subset  S_d\subset S_{d-1}\subset \cdots \subset S_0=\widehat{\Ac}, $$  
 and that for  $0\le i \le d $ 
there  is a Hilbert space $\Hc_i $, and for every $ \gamma\in \Gamma_i $  a concrete realization $ (\pi_\gamma,\H_i) $ of 
$\gamma $ on the Hilbert space $ \H_i $. 
The set $ S_d $ is the collection of all 
characters of $\Ac$.

We say that the $C^* $-algebra $\Ac$ has  \textit{norm controlled dual limits} if for every 
 $ a\in \Ac $ one has
\begin{enumerate}
\item\label{nrocontspec_1}  The mappings $ \gamma\to \Fc (a)(\gamma) $ are norm continuous  
on the difference sets $ \Gamma_i = S_{i-1}\setminus S_i $.
\item\label{nrocontspec_2}    For any  $ i=0,\dots, d+1$  and for any 
converging sequence contained in $ \Gamma_i $ with limit set  outside $ \Gamma_i $, hence in $S_i$, 
there exists a properly converging sub-sequence $\overline{\gamma}=(\gamma_k)_{k\in\NN} $,  a constant  $ C>0 $ 
and for every $ k\in\NN $  an involutive  linear mapping $ \tilde\sigma_{\overline{\gamma},k}: CB(S_{i})\to \Bc(\H_i)$, 
that is  bounded by $ C\norm{\cdot}_{S_{i}} $, such that
$$
(\forall a\in\Ac)\quad \lim_{k\to\infty}\norm{\Fc (a)(\gamma_k)-
\tilde\sigma_{\overline\gamma,k} (\Fc (a)\vert_{S_{i}})}_{\Bc(\Hc_i)}=0.
$$
 \end{enumerate}
 \end{definition}

\begin{theorem}\label{newcor}
Assume the $C^*$-algebra $\Ac$ satisfies 
the conditions Theorem~\ref{S5}.  Then $\Ac$ has norm controlled dual limits. 
Specifically, with the notations of Theorem~\ref{S5}, 
let 
$1\le \ell \le n$ be fixed, and  
$\bar{\gamma}=(\gamma_k)_{k} \in \Gamma_\ell$ be a properly convergent  sequence with limit set outside $\Gamma_\ell$. 
Then there exists a sequence   $(\sigma_{\bar{\gamma}, k})_k$ 
of completely positive and completely contractive maps 
$\sigma_{\bar{\gamma}, k}\colon CB(\hat{\Ac}\setminus V_{\ell})  
\to \Bc(\Hc_\ell)$ 
such that 
$$(\forall a\in\Ac)\quad  
\lim_{k\to \infty} \norm{\Fc(a) (\gamma_k)- \sigma_{\bar{\gamma}, k}(\Fc(a)\vert_{\hat{\Ac}\setminus V_{\ell}})}_{\Bc(\Hc_{\ell})}=0.$$
\end{theorem}

\begin{proof}
From Theorem~\ref{S5} we see that there is a linear map 
$$\nu_\ell\colon \Ran\Fc_{\Ac/{\Jc_{\ell}}} \to \Cc_b(\Gamma_{\ell}, \Kc(\Hc_{\ell}))$$
that is completely positive, completely isometric and almost $*$-homomorphism, and 
such that 
\begin{equation}\label{newcor_proof_eq1}
(\forall a\in\Ac)\quad 
 \Fc(a)\vert_{\Gamma_\ell}  - \nu_\ell\big(\Fc( a)\vert_{\hat{\Ac}\setminus V_{\ell}})\in  \Cc_0(\Gamma_\ell, \Kc (\Hc_{\ell})).
\end{equation}
Define now $\tilde{\sigma}_{\bar{\gamma}, k}\colon \Fc_{\Ac/\Jc_{\ell}} (\Ac/{\Jc_{\ell}} )\to \Bc(\Hc_{\ell})$, 
$$\tilde{\sigma}_{\bar{\gamma}, k} (\cdot):= \bigl(\nu_\ell(\cdot)\bigr)(\gamma_k),$$
for every $k\in\NN$. 
These maps are completely positive, and by \cite[Thm.~1.2.3]{Ar69} and \cite[Lemma~3.9]{CE76},  
they extend from the $C^\ast$-subalgebra 
$\Ran \Fc_{\Ac/\Jc_{\ell}}\subset CB(\hat{\Ac}\setminus V_\ell)$ to completely positive and completely contractive linear maps 
$$\sigma_{\bar{\gamma}, k}\colon CB(\hat{\Ac}\setminus V_\ell) \to \Bc(\Hc_{\ell}).$$ 
Due to the properties of $\nu_\ell$,  
and using~\eqref{newcor_proof_eq1} via \cite[rem. II, p. 474]{Fe62} 
which implies $\gamma_k\to\infty$ in $\Gamma_\ell$, 
the sequence $(\sigma_{\bar{\gamma}, k})_k$ satisfies all the properties in the statement. 
\end{proof}

\subsection{Fourier transforms of $C^*$-algebras with norm controlled dual limits}

\begin{definition}[see also \cite{LuRe15}]\label{dnormconspec} 
Let $\emptyset= S_{d+1}\subset S_d\subset\cdots  \subset S_0=S $ be a locally compact topological space of step $\leq d $. 
Choose  for every $i=1,\dots, d $  a Hilbert space $\Hc_i $ and assume that $\Hc_{d}=\CC $. 

 Let $ B^*(S) $ be the set of all operator fields $ \varphi $ defined over $ S $ such that
\begin{enumerate}
\item\label{dnormconspec_1} $ \varphi(\gamma)\in \Kc(\H_i) $ for every $ \gamma\in\Gamma_i= S_{i-1}\setminus S_{i}$, 
$i=1,\dots, d $.
\item \label{dnormconspec_2} The field $ \varphi $ is uniformly bounded, that is,  we have that
$$
 \norm{\varphi} =\sup\left\{\norm{\varphi(\gamma)}_{\Bc(\Hc_i)}\mid \gamma\in \Gamma_i, \, i= 1, \dots ,d\right\} <\infty.
$$
\item \label{dnormconspec_3} The mappings $ \gamma\to \varphi(\gamma) $ are norm continuous on the difference sets $ \Gamma_i $.
\item   \label{dnormconspec_4} We have for any sequence $ (\gamma_k)_{k\in\NN} \subset S$ going to infinity, 
that  
$$ \underset{k\to\infty}{\lim}\norm{\varphi(\gamma_k)}_{\text{op}}=0.$$
\item \label{dnormconspec_5}  For any  $ i=1,\dots, d+1$,   and for any 
converging sequence contained in $ \Gamma_i= S_{i-1}\setminus S_i$ with limit set  outside $ \Gamma_i $, 
there exists a properly converging sub-sequence $\overline{\gamma}=(\gamma_k)_{k\in\NN} $,  a constant  $ C>0 $ 
and for every $ k\in\NN $  an involutive  linear mapping $\tilde\sigma_{\overline{\gamma},k}\colon  CB(S_{i})\to \Bc(\Hc_i)$, 
which is  bounded by $ C\norm{.}_{S_{i}} $, such that
$$
\lim_{k\to\infty}\norm{\varphi (\gamma_k)-\tilde\sigma_{\overline\gamma,k} (\varphi\vert_{ S_{i}})}_{\Bc(\Hc_i)}=0.
$$
 \end{enumerate}
\end{definition}
 
\begin{theorem}[see also \cite{LuRe15}]\label{Th2.11} 
Let $ S $ be a locally compact topological space of step $\leq d $.  
Then the set $B^*(S) $ of Definition~\ref{dnormconspec} is  a closed involutive  subspace of $\ell^\infty(S) $.
Furthermore $B^*(S) $is a  $ C^*$-subalgebra of $ \ell^\infty(S)$ with spectrum $ S$
if and only if  
 all  the mappings $ \tilde \sigma_{\overline{\gamma},k}$ are almost homomorphisms, i.e.,  
$$
 \lim_{k\to\infty}\norm{\tilde \sigma_{\overline{\gamma},k}(\varphi\cdot\psi)-\tilde \sigma_{\overline{\gamma},k}
(\varphi)\cdot \tilde \sigma_{\overline{\gamma},k}(\psi)}_{\Bc(\Hc_i)}=0, \quad \varphi,\, \psi\in B^*(S),
$$
and the restrictions $ B^*(S)\vert_{S_{i-1}}$ contain the spaces $ C_0(\Gamma_i, \Hc_i)$, $i=1,\dots, d+1$.
\end{theorem}

\begin{proof}
We easily see that the conditions \eqref{dnormconspec_1} -- \eqref{dnormconspec_4} in Definition~\ref{dnormconspec}
 imply that $ B^*(S) $ is a closed involution-invariant subspace 
of $\ell^\infty(S) $.

For $ i=1,\dots, d+1$,  let  $ B^*_i $ be the set of all operator fields defined over 
$ S_{i-1} $, satisfying conditions \eqref{dnormconspec_1}--\eqref{dnormconspec_5}    on the sets  $S_j$, $j=d,\dots, {i-1}$. 
Then obviously for every $ i$ the restriction $ B^*(S)\vert_{S_{i-1}}$ of the space $B^*(S)$ to $ S_{i-1}$ is contained in $ B^*_i$. 
and  $ B^*(S)$ is also an algebra and hence a $ C^*$-subalgebra of $ \ell^\infty(S)$. 

Let us show that the spectrum of $ B^*(S)$ can be identified with the space $ S$. 
Since $ C_0(\Gamma_i,\Hc_i)$ is contained in $ B^*_i$ for every $ i$, it follows that the representations 
$ \pi_s\colon \varphi\to \varphi(s)\in\Kc(\Hc_i)$ of $B^*(S) $ are irreducible.
It follows from the choice of $ \Hc_d$ and the properties of $B^*(S) $, that $ B^*_d=B^*(S)\vert_{S_d}=C_0(S_d)$. 

Suppose that  for some $ 0\le i<  d$ the spectrum of the algebra $ B^*_{i+1}$ is the space $ S_{i}$.
 Let $ \pi\in \widehat{B^*_i} $. 
 Consider the kernel $ K_{i+1} $ of  the restriction mapping $R_{i+1}$
from $ B^*_i $ into $ l^\infty(S_{i}) $. 
If $ \pi(K_{i+1})=\{0\} $, then we can consider $ \pi $ 
as being a representation of the quotient algebra $ B^*_i/K_{i+1} $. 
But the image $ \Bc^*(S)\vert_{S_{i}}$ of $ R_{i+1} $ is a $ C^*$-subalgebra of   $ B^*_{i+1} $, 
the spectrum of $ B^*_{i+1}$ is by assumption the set $ S_{i}$  and $S_{i}$ is also contained in 
the spectrum of the subalgebra $ \Bc^*(S)\vert_{S_{i}} $. 
Hence by the Stone-Weierstrass theorem 
\cite[Th. 11.1.8]{Dix64} the algebras $ B^*_{i+1}$ and $ B^*(S)\vert_{S_{i}}$ coincide. 
Hence $ \pi$ is an evaluation at a point in $ S_{i}$.

If $ \pi(K_{i+1})\ne\{0\} $, then we look at the restriction of $ \pi $ to this ideal. 
The elements in $ K_{i+1} $ are operator fields defined on $ S_{i-1} $ which are norm continuous, 
which go to 0 at infinity and by condition~\eqref{dnormconspec_5} in Definition~\ref{dnormconspec}, 
for any properly converging sequence 
$ \overline{\gamma}\subset \Gamma_i $ with limit  outside $ \Gamma_i $, for every $ \varphi\in K_{i+1} $, we have that
$$
\lim_{k\to\infty}\norm{\varphi(\gamma_k)}_{\Bc(\Hc_i)}
=
\lim_{k\to\infty}\norm{\varphi(\gamma_k)-\tilde\sigma_{\overline{\gamma},k}(\varphi\vert_{{S_{i}}})}_{\Bc(\Hc_i)}=0.
$$
 This shows that $ K_{i+1}\subset C_0(\Gamma_i,\Kc(\Hc_i)) $. 
 Since by condition \eqref{dnormconspec_1} in Definition~\ref{dnormconspec} 
we know that $ C_0(\Gamma_i,\Kc(\Hc_i)) \subset K_{i+1}$ it follows that $ C_0(\Gamma_i,\Kc(\Hc_i)) = K_{i+1}$ 
and then $ \pi$ is an evaluation at  an element $ s\in \Gamma_i$. 
 Finally the spectrum of $ B^*_i$ is the set $ S_{{i-1}}$.  
Therefore again by the aforementioned Stone-Weierstrass theorem the algebras $ B^*_i $ and  $B^*(S)\vert_{S_{i-1}}$  coincide. 

If $ B^*(S) $ is a $ C^* $-algebra, then obviously the conditions  of Definition~\ref{dnormconspec} are fulfilled. 
 \end{proof}

\begin{corollary}\label{disb}
Let $ \Ac $ be a $ C^* $-algebra with norm controlled dual limits. Then the 
Fourier transform of $ \Ac $ is the $ C^* $-algebra $ B^*(\widehat{\Ac}) $. 
 \end{corollary}

\begin{proof}
Use Theorem~\ref{Th2.11}.
\end{proof}

\section{Application to $C^*$-algebras of nilpotent Lie groups}\label{section4}

We prove in this section that the $C^*$-algebras of nilpotent Lie groups are special solvable, and this allow us to 
use the results of the previous section to describe the image of their operator-valued Fourier transform.

Throughout this section we denote by $\gg$ any nilpotent Lie algebra with its corresponding Lie group $G=(\gg,\cdot)$ and 
with a fixed Jordan-H\"older sequence 
$$\{0\}=\gg_0\subseteq\cdots\subseteq\gg_m=\gg$$
and we pick $X_j\in\gg_j\setminus\gg_{j-1}$ for $j=1,\dots,m$. 
We denote by 
$$\langle\cdot,\cdot\rangle\colon\gg^*\times\gg\to\RR$$
the duality pairing between $\gg$ and its linear dual space~$\gg^*$, 
and for every subalgebra $\hg\subseteq\gg$ we define 
$$(\forall\xi\in\gg^*)\quad \hg(\xi):=\{X\in\hg\mid(\forall Y\in\hg)\  \langle\xi,[X,Y]\rangle=0\}.$$
We denote by $\Ec$ the set of all subsets of $\{1,\dots,m\}$ 
endowed with the total ordering defined for all $e_1\ne e_2$ in $\Ec$ by 
$$e_1\prec e_2\iff\min(e_1\setminus e_2)<\min(e_2\setminus e_1)$$
where we use the convention $\min\emptyset=\infty$, so in particular $\max\Ec=\emptyset$.  
This is precisely the ordering introduced in \cite[subsect. 3.4, p. 454]{Pe84}. 
We also endow the $m$-th Cartesian power $\Ec^m$ with the total ordering 
obtained from 
the above ordering $\prec$ 
as in~\cite[Def. 1.2.5]{Pe94}. 

\subsection{Coarse stratification of $\gg^*$ and continuity of trace on $\widehat{G}$}
We define the jump indices
$$(\forall \xi\in\gg^*)\quad J_\xi:=\{j\in\{1,\dots,m\}\mid \gg_j\not\subset\gg(\xi)+\gg_{j-1}\}$$
and 
$$(\forall e\in\Ec)\quad \Omega_e:=\{\xi\in\gg^*\mid J_\xi=e\}.$$
The \emph{coarse stratification of $\gg^*$} is the family $\{\Omega_e\}_{e\in\Ec}$, 
which is a finite partition of~$\gg^*$ consisting of $G$-invariant sets. 
For every coadjoint $G$-orbit $\Oc\in\gg^*/G$ we define $J_{\Oc}:=J_\xi$ for any $\xi\in\Oc$ and then 
\begin{equation}\label{coarse_eq1}
(\forall e\in\Ec)\quad \Xi_e:=\{\Oc\in\gg^*/G\mid J_{\Oc}=e\}.
\end{equation}

\begin{lemma}\label{coarse_lemma}
Assume the above setting and for every $\Oc\in\gg^*/G$ pick any unitary irreducible representation 
$\pi_{\Oc}\colon G\to \Bc(\Hc_{\Oc})$ which is associated with~$\Oc$ via Kirillov's correspondence. 
If we endow the space $\gg^*/G\simeq\widehat{G}$ with its canonical topology, 
then for every index set $e\in\Ec$ the following assertions hold: 
\begin{enumerate}
\item\label{coarse_lemma_item1} The relative topology of $\Xi_e\subseteq\gg^*/G$ is Hausdorff. 
\item\label{coarse_lemma_item2} For every test function $\phi\in\Cc_0^\infty(G)$ 
the function 
$$\Xi_e\to\CC,\quad \Oc\mapsto\Tr(\pi_{\Oc}(\phi)) $$
is well defined and continuous. 
\end{enumerate}
\end{lemma}

\begin{proof}
The first assertion follows as an application of \cite[Th. 3.1.14(iv)]{CG90} 
for the coadjoint action of $G$, 
which provides a homeomorphism of $\Xi_e$ onto a certain algebraic subset 
of the vector space $\spa\{X_j\mid j\in\{1,\dots,m\}\setminus e\}$. 

The second assertion is just~\cite[Lemma 4.4.4]{Pe84}. 
\end{proof}

\subsection{Fine stratification of $\gg^*$ and continuity of operator fields on $\widehat{G}$}

We define 
$$(\forall \xi\in\gg^*)(\forall k=1,\dots,m)\quad J_\xi^k:=\{j\in\{1,\dots,k\}\mid \gg_j\not\subset\gg_k(\xi)+\gg_{j-1}\}$$
and 
$$(\forall \varepsilon\in\Ec^m)\quad \Omega_\varepsilon:=\{\xi\in\gg^*\mid (J_\xi^1,\dots,J_\xi^m)=\varepsilon\}.$$
The \emph{fine stratification of $\gg^*$} is the family $\{\Omega_\varepsilon\}_{\varepsilon\in\Ec^m}$, 
which is again a finite partition of~$\gg^*$ consisting of $G$-invariant sets. 
For every coadjoint $G$-orbit $\Oc\in\gg^*/G$ we define $J_{\Oc}^k:=J_\xi^k$ for any $\xi\in\Oc$ and $k=1,\dots,m$, 
and then 
$$(\forall \varepsilon\in\Ec^m)\quad \Xi_\varepsilon:=\{\Oc\in\gg^*/G\mid (J_{\Oc}^1,\dots,J_{\Oc}^m)=\varepsilon\}.$$
For every $\varepsilon\in\Ec^m$ we also define $\Gamma_\varepsilon\subseteq\widehat{G}$ as the image of $\Xi_\varepsilon$ 
through Kirillov's correspondence $\gg^*/G\simeq\widehat{G}$, 
which is actually a homeomorphism. 

\begin{remark}\label{fine_rem}
\normalfont
Every stratum of the coarse stratification of $\gg^*$ 
is the disjoint union of a few strata of the fine stratification of~$\gg^*$. 
More precisely, since $\gg_m=\gg$, we have 
$$(\forall e\in\Ec)\quad \Omega_e=\{\xi\in\gg^* \mid J_\xi^m=e\}
=\bigsqcup_{\varepsilon\in\Ec^{m-1}\times\{e\}}\Omega_\varepsilon. $$
\qed
\end{remark}

\begin{lemma}\label{fine_lemma2}
For every $\varepsilon\in\Ec^m$ there exists a Hilbert space $\Hc_\varepsilon$ 
and every equivalence class of representations $\gamma\in\Gamma_\varepsilon$ 
has a representative $\pi_\gamma\colon G\to\Bc(\Hc_\varepsilon)$ 
with the property that 
 for every $a\in C^*(G)$ the map 
$\Pi_a\colon \Gamma_\varepsilon\to\Bc(\Hc_\varepsilon)$, $\gamma\mapsto\pi_\gamma(a)$, is norm continuous.
\end{lemma}

\begin{proof}
By using \cite[Cor. 2.15]{LiRo96}, 
we obtain a Hilbert space $\Hc_\varepsilon=L^2(\RR^d)$ 
and for every $\gamma\in\Gamma_\varepsilon$ 
a representative $\pi_\gamma\colon G\to\Bc(\Hc_\varepsilon)$ 
with the property that 
for all $\phi\in\Cc^\infty_0(G)$, $f_1\in\Cc^\infty_0(\RR^d)$, and $f_2\in\Sc(\RR^d)$ with $\widehat{f}_2\in\Cc^\infty_0(\RR^d)$ 
the function 
$\Gamma_\varepsilon\to\CC$, $\gamma\mapsto\langle\pi_\gamma(\phi)f_1,f_2\rangle$ 
is continuous. 

On the other hand, it follows by Lemma~\ref{coarse_lemma}\eqref{coarse_lemma_item2} 
that for every $\phi\in\Cc^\infty_0(G)$ the function $\Gamma_\varepsilon\to\CC$, $\gamma\mapsto\Tr\pi_\gamma(\varphi)$, 
is continuous. 
Hence we may use Lemma~\ref{F2} with $\Sc=\Cc^\infty_0(G)$ to obtain the conclusion. 
\end{proof}

\begin{lemma}\label{fine_lemma1}
For every $\varepsilon\in\Ec^m$ 
the set $\Xi_\varepsilon$ is dense and  open  in  $(\gg^*/G)\setminus\bigcup\limits_{\Ec^m\ni\delta\prec\varepsilon}\Xi_\delta$.
 \end{lemma}

\begin{proof}
If $e\in\Ec$ is the set for which $\varepsilon\in\Ec^{m-1}\times\{e\}$, 
then Remark~\ref{fine_rem} implies in particular that $\Xi_\varepsilon\subseteq\Xi_e$. 
It then follows by Lemma~\ref{coarse_lemma}\eqref{coarse_lemma_item1} that the relative topology of $\Xi_\varepsilon$ is Hausdorff. 

On the other hand, it 
follows by 
\cite[Prop. 1.3.2]{Pe94} 
that $\Omega_\varepsilon$ 
is a Zariski-open subset of $\gg^*\setminus\bigcup\limits_{\Ec^m\ni\delta\prec\varepsilon}\Omega_\delta$. 
Since the quotient map $\gg^*\to\gg^*/G$ is an open map 
(see \cite[Ch. 1, \S 5, no. 2, Ex. 1]{Bo07}), 
it then follows that $\Xi_\varepsilon$ is a dense open subset of $(\gg^*/G)\setminus\bigcup\limits_{\Ec^m\ni\delta\prec\varepsilon}\Xi_\delta$.
\end{proof}

\subsection{Image of the Fourier transform of $\Ac=C^*(G)$} 

\begin{proposition}\label{prop3.5}
Define the sequence 
$$\emptyset=V_0\subsetneqq V_1\subsetneqq\cdots\subsetneqq V_n=\widehat{G} $$
by the conditions
\begin{itemize}
\item $\{\emptyset\}\cup\{V_j\setminus V_{j-1}\mid 1\le j\le n\}=\{\Gamma_\varepsilon\mid\varepsilon\in\Ec^m\}$; 
\item if $j_1,j_2\in\{1,\dots,n\}$ and $\varepsilon_1,\varepsilon_2\in\Ec^m$ 
with $V_{j_1}\setminus V_{j_1-1}=\Gamma_{\varepsilon_1}$ and  $V_{j_2}\setminus V_{j_2-1}=\Gamma_{\varepsilon_2}$, 
then we have $j_1<j_2$ if and only if $\varepsilon_1\prec\varepsilon_2$. 
\end{itemize}
Then $V_j$ is an open subset of $\widehat{G}$ for all $j=1,\dots,n$, 
and both hypotheses \eqref{S5_item1}--\eqref{S5_item2} of Theorem~\ref{S5} are satisfied. 
\end{proposition}

\begin{proof}
If $j=1,\dots,n$, then $V_j$ is an open subset of $\widehat{G}$ as a consequence of Lemma~\ref{fine_lemma1}, 
since Kirillov's correspondence $\gg^*/G\simeq\widehat{G}$ is a homeomorphism. 

Moreover, to see that the other hypotheses of Theorem~\ref{S5} 
are satisfied, we just apply Lemma~\ref{fine_lemma1} again (for hypothesis~\eqref{S5_item1}), 
and Lemma~\ref{fine_lemma2} (for hypothesis~\eqref{S5_item2}). 
This completes the proof. 
\end{proof}

We have thus obtained the next theorem. 

 \begin{theorem}\label{nilpoare}
Let $G $ be  connected simply connected nilpotent Lie group.
Then the $C^*$-algebra $C^*(G) $ has norm controlled dual limits.
 \end{theorem}
 
\begin{proof}
The theorem follows from Proposition~\ref{prop3.5}, Theorem~\ref{S5}, and Theorem~\ref{newcor}.
\end{proof}

\subsection{The index of a nilpotent Lie group}

\begin{definition}
\normalfont
Let $G$ be any simply connected Lie group with its Lie algebra $\gg$. 
The \emph{index} of $G$ (respectively, of $\gg$) is defined as 
$$\ind G=\ind\gg:=\dim\gg-\max_{\xi\in\gg^*}\dim\Ad^*_G(G)\xi.$$
\end{definition}

\begin{remark}\label{ind_rem}
\normalfont
For every $\xi\in\gg^*$, the coadjoint action gives rise to a $G$-equivariant diffeomorphism 
$G/G(\xi)\simeq\Ad^*_G(G)\xi$, hence $\dim\Ad^*_G(G)\xi=\dim\gg-\dim G(\xi)$ 
and it follows that 
$$\ind G=\ind\gg=\min_{\xi\in\gg^*}\dim G(\xi).$$
If $Z$ is the center of $G$, then $G(\xi)\supseteq Z$ for every $\xi\in\gg^*$, hence $\ind G\ge\dim Z$. 
\qed
\end{remark}

\begin{proposition}\label{ind_prop}
Let $G$ be any simply connected Lie group with its Lie algebra $\gg$ 
If $G$ is nilpotent, then the following assertions hold: 
\begin{enumerate}
\item\label{ind_prop_item1} 
If $e_1\prec\cdots\prec e_n\prec\emptyset$ are all the index sets of the coadjoint orbits of $G$ 
with respect to some Jordan-H\"older basis of $\gg$, then 
$\ind G=\dim\gg-\card e_1$. 
\item\label{ind_prop_item2} 
One has $\ind G=r$ if and only if there exists some open dense subset $V\subseteq\widehat{G}$ 
that is homeomorphic to some open subset of $\RR^r$. 
\end{enumerate}
\end{proposition}

\begin{proof}
\eqref{ind_prop_item1} 
Recall that $\card e_1=\max\limits_{\xi\in\gg^*}\dim \Ad^*_G(G)\xi$, hence $\ind G=\dim\gg-\card e_1$.  

\eqref{ind_prop_item2} 
It follows by Lemma~\ref{dom} that it suffices to check the equality $\ind G=r$ 
for any particular choice of the open dense subset $V\subseteq\widehat{G}$ 
that is homeomorphic to some open subset of $\RR^r$. 
Using the above notation and the one introduced in \eqref{coarse_eq1}, 
if $V\subseteq\widehat{G}$ is the image of $\Xi_{e_1}\subseteq\gg^*/G$ by Kirillov's correspondence, 
then $V$ is open and dense in $\widehat{G}$ by \cite[Cor. 3.4.2]{Pe84} or \cite[Cor. 1.3.2]{Pe94}, 
and also $V$ is homeomorphic to some Zariski open subset of $\RR^r$, where $r:=\dim\gg-\card e_1$ 
(see \cite[Subsect 1.2(c)--(d)]{Pe84}). 
For these $V$ and $r$ one has $\ind G=r$  by Assertion~\eqref{ind_prop_item2}. 
\end{proof}

\begin{remark}
\normalfont
It follows by Proposition~\ref{ind_prop}\eqref{ind_prop_item1} 
that the number $\card e_1$ is independent on the choice of the Jordan-H\"older basis of $\gg$. 
\qed
\end{remark}

\subsection{$C^*$-algebras of nilpotent Lie groups are special solvable}

\begin{theorem}\label{nilpsolv}
Let $G$ be any connected, simply connected, nilpotent Lie group with its Lie algebra~$\gg$. 
Then $C^*(G)$ is a special solvable $C^*$-algebra 
with $\ind C^*(G)=\ind G$ and $C^*(G)/\Comm(C^*(G))\simeq\Cc_0([\gg,\gg]^\perp)$. 
\end{theorem}

\begin{proof}
This follows by Propositions \ref{prop3.5} and \ref{ind_prop}, 
using also \cite[Cor. 3.4.2]{Pe84}. 
\end{proof}

\section{Uniqueness of Heisenberg group via special solvable $C^*$-algebras}\label{section5}

We now prove that the Heisenberg groups are the only nilpotent Lie groups whose index is~1 and whose $C^*$-algebras 
are special solvable of length~2. 

\begin{proposition}\label{exh}
Let $G$ be any nilpotent Lie group for which there exists a short exact sequence of $C^*$-algebras 
$$0\to\Cc_0(\Gamma_1,\Kc(\Hc))\to C^*(G)\to\Cc_0([\gg,\gg]^\perp)\to0$$ 
satisfying the following conditions: 
\begin{itemize}
\item $\Gamma_1$ is a dense open $\RR$-subspace of $\widehat{G}$ that is homeomorphic to $\RR^*$; 
\item $\Hc$ is a separable infinite-dimensional complex Hilbert space. 
\end{itemize}
Then there exists a unique integer $d\ge 1$ such that $\dim[\gg,\gg]^\perp=2d$ and 
$G$ is isomorphic to the Heisenberg group $\HH_{2d+1}$.  
\end{proposition}

\begin{proof}
Let us denote by $\Oc_\xi$ the coadjoint orbit of every $\xi\in\gg^*$. 
Since the group $G$ has infinite-dimensional representations, it follows that it is non-commutative, 
hence there exists $\xi_1\in\gg^*$ with $\Oc_{\xi_1}\ne\{\xi_1\}$, and then  
$\Oc_{t\xi_1}\ne\{t\xi_1\}$ for every $t\in\RR^*$.  
Identifying $\widehat{G}$ with the space of coadjoint orbits $\gg^*/G$ by Kirillov's correspondence, 
it follows by the exact sequence from the statement that one has the disjoint union  
$\gg^*/G=\Gamma_1\sqcup[\gg,\gg]^\perp$. 
Moreover, $\xi_1\in\Gamma_1$, and by the hypothesis that $\Gamma_1$ is open $\RR$-subspace of $\widehat{G}$ 
that is homeomorphic to $\RR^*$ 
we obtain $\Gamma_1=\{\Oc_{t\xi_1}\mid t\in\RR^*\}$, hence 
\begin{equation}\label{exh_proof_eq1}
\gg^*/G=\{\Oc_{t\xi_1}\mid t\in\RR^*\}\sqcup[\gg,\gg]^\perp.
\end{equation}
Denoting the center of $\gg$ by $\zg$ and using the fact that the Lie algebra $\gg$ is nilpotent, 
it follows that $[\gg,\gg]\cap\zg\ne\{0\}$. 
Then, using \eqref{exh_proof_eq1} and reasoning by contradiction, 
it easily follows that there exist $x,y\in\gg$ with $z:=[x,y]\in\zg\setminus\{0\}$ 
and $\langle\xi_1,z\rangle\ne0$. 

We claim that $[\gg,\gg]=\RR z$. 
In fact, if $[\gg,\gg]\supsetneqq\RR z$, then there exists $\xi\in\gg^*$ with $\langle\xi,z\rangle=0$ 
and $\langle\xi,[\gg,\gg]\rangle\ne\{0\}$. 
Then we have $\xi\not\in[\gg,\gg]^\perp$, and 
also $\xi\not\in\bigcup\limits_{t\in\RR^*}\Oc_{t\xi_1}$ 
since $\langle\xi,z\rangle=0\not\in\langle \bigcup\limits_{t\in\RR^*}\Oc_{t\xi_1},z\rangle$. 
We thus obtained a contradiction with \eqref{exh_proof_eq1}, and this proves the above claim. 

Since $[\gg,\gg]=\RR z$, it follows that there exist  some uniquely determined integers $d\ge 1$ and $k\ge 0$ 
with $\gg=\hg_{2d+1}\times\ag_k$, where $\ag_k$ is the abelian $k$-dimensional Lie algebra. 
Then the index of $\gg$ is $\ind\gg=k+1$ by Remark~\ref{ind_rem}. 
On the other hand, since $\Gamma_1$ is a dense open subset of  $\widehat{G}$ that is homeomorphic to $\RR^*$, 
it follows by Proposition~\ref{ind_prop}\eqref{ind_prop_item2} that $\ind\gg=1$, 
hence by comparing with the above formula we obtain $k=0$, that is, $\gg=\hg_{2d+1}$, 
and this completes the proof. 
\end{proof}

\section{On the conjecture on continuous-trace subquotients}\label{section6}

One of the open problems in the theory of $C^*$-algebras of exponential Lie groups 
is the conjecture that every continuous-trace subquotient of their $C^*$-algebra 
has its Dixmier-Douady invariant equal to zero \cite{RaRo88, Ro94, LiRo96}. 
It was proved to be true so far only for 2-step nilpotent Lie groups \cite[Th. 3.4]{LiRo96}.
Using a different method,  based on the topology of the space of coadjoint orbits,  
we will also prove it here for uncountable families of 3-step nilpotent Lie groups 
and also for a sequence of nilpotent Lie groups of arbitrarily high nilpotency step. 
\begin{notation}
\normalfont
For any separable $C^*$-algebra $\Ac$ let $\Subquot(\Ac)$ be the set of all $*$-isomorphism classes of subquotients of $\Ac$, 
that is, $C^*$-algebras $\Jc_2/\Jc_1$, 
where $\Jc_1\subseteq\Jc_2$ are ideals of $\Ac$. 
Let $\Subquot^{\Hausd}(\Ac)$ be the $*$-isomorphism classes of subquotients of $\Ac$ with Hausdorff dual. 
We also denote by $\Subquot^{\Tr}(\Ac)$ the set of $*$-isomorphism classes of subquotients of $\Ac$ with continuous trace, 
and by $\Subquot^{\Tr}_0(\Ac)$ its subset corresponding to the subquotients whose Dixmier-Douady invariants are equal to zero. 

Then one has 
$$\Subquot^{\Tr}_0(\Ac)\subseteq\Subquot^{\Tr}(\Ac)\subseteq\Subquot^{\Hausd}(\Ac)\subseteq\Subquot(\Ac).$$

We denote by $[\Ac]$ the $*$-isomorphism class of $\Ac$. 
\end{notation}

For later use note that from \cite[Th.~10.9.3]{Dix64} that if $\Xc$ is any separable $C^*$-algebra with continuous trace, then the Dixmier-Douady invariant of $\Xc$ vanishes if and only if $\Xc$ is isomorphic to the $C^*$-algebra of a continuous field of non-zero Hilbert spaces on $\widehat{\Xc}$, hence this property is invariant under direct sums of $C^*$-algebras. 

Recall that a separable $C^*$-algebra has its Dixmier-Douady invariant equal to zero if and only it is Morita equivalent to a commutative $C^*$-algebra or, equivalently,  if it is stably isomorphic to a commutative $C^*$-algebra (see for instance \cite[Sect.~5.3]{Ec96} and the references therein).

The conjecture below has been raised in \cite[Sect. 4]{RaRo88}, \cite[Sect. 3]{Ro94}, and \cite[Conj. 3.2]{LiRo96}, 
and so far it was proved only for the 2-step nilpotent Lie groups (see \cite[Th. 3.4]{LiRo96}).

\begin{conjecture}\label{SQ0}
\normalfont
If $\Ac$ is the $C^*$-algebra of an arbitrary exponential Lie group, then $\Subquot^{\Tr}_0(\Ac)=\Subquot^{\Tr}(\Ac)$. 
\end{conjecture}

It follows by Lemma~\ref{SQ2} below that the class of $C^*$-algebras $\Ac$ with $\Subquot^{\Tr}_0(\Ac)=\Subquot^{\Tr}(\Ac)$ 
is closed with respect to taking ideals and quotients, hence also subquotients. 
Thus Conjecture~\ref{SQ0} 
is equivalent to the formally stronger assertion that $\Subquot^{\Tr}_0(\Ac)=\Subquot^{\Tr}(\Ac)$ 
whenever $\Ac$ is any subquotient of the $C^*$-algebra of an arbitrary exponential Lie group.

\subsection*{Continuous-trace subquotients and $C^*$-algebra extensions}

\begin{lemma}\label{SQ2}
For any short exact sequence of separable $C^*$-al\-ge\-bras 
\begin{equation}\label{SQ1_eq1}
0\to\Jc\to\Ac\mathop{\to}\limits^p\Bc\to0
\end{equation}
we have
$
\Subquot(\Jc)\cup \Subquot(\Bc)\subseteq \Subquot(\Ac)$.
\end{lemma}

\begin{proof}
Since $\Jc$ is an ideal of $\Ac$, 
it follows by \cite[Prop. 1.8.5]{Dix64} that any ideal of $\Jc$ is an ideal of $\Ac$, 
and hence $\Subquot(\Jc)\subseteq\Subquot(\Ac)$. 

Now let $\Jc_1\subseteq\Jc_2$ be any ideals of $\Bc$. 
Then $p^{-1}(\Jc_1)\subseteq p^{-1}(\Jc_2)$ are ideals of $\Ac$ 
and $p$ gives rise to a $*$-isomorphism $p^{-1}(\Jc_k)/\Jc\simeq\Jc_k$ for $k=1,2$, 
hence to a $*$-isomorphism $p^{-1}(\Jc_2)/p^{-1}(\Jc_1)\simeq\Jc_2/\Jc_1$. 
Thus $[\Jc_2/\Jc_1]=[p^{-1}(\Jc_2)/p^{-1}(\Jc_1)]\in\Subquot(\Ac)$, 
and this completes the proof of the inclusion $\Subquot(\Bc)\subseteq \Subquot(\Ac)$. 
\end{proof}

\begin{remark}\label{S33}
\normalfont
For any topological space $T$ we denote by $\LC(T)$ the set of all locally closed subsets of $T$.
We recall from Remark~\ref{S3} that
for any $C^*$-algebra $\Ac$ 
the locally closed subsets of $\widehat{\Ac}$ 
are precisely the spectra of subquotients of $\Ac$.  
The map 
$$\Psi_\Ac\colon \Subquot(A)\to\LC(\widehat{\Ac}),\quad [\Jc_2/\Jc_1]\mapsto \widehat{\Jc_2/\Jc_1}$$
is a well-defined bijection. 
\qed
\end{remark}

\begin{lemma}\label{SQ5}
Let $X$ be a first-countable topological space.  
Assume that $X= X_1 \sqcup X_2$ where 
\begin{enumerate}[(i)] 
\item $X_1$ is open, locally compact Hausdorff,  and its points are separated in $X$. 
\item If $\overline{x}=\{x_j\}_{j\in \NN}$ is a convergent sequence contained in $X_1$ and such that 
its set of limit points $L(\overline{x})$ is contained in  $X_2$, then $L(\overline{x})$ has no isolated points. 
\end{enumerate}
Then for every locally closed and Hausdorff subset $S$ of $X$ we have that $S\cap X_1$ is relatively closed in $S$. 
\end{lemma}

\begin{proof}
Note first that the hypotheses imply that for  a convergent sequence  $\overline{x}=\{x_j\}_{j\in \NN}$ contained in $X_1$, we have:
Either 
the set of its limit points 
$L(\overline{x})$ is a single point in $X_1$, which is Hausdorff, or, 
because the points of $X_1$ are separated in $X$, $L(\overline{x})\subseteq X_2$, hence $L(\overline{x})$ has no isolated points.

Let $y\in S\cap \overline{S\cap X_1}$. 
We have to prove that $y \in S\cap X_1$. 
Assume the contrary, that is, $y \not \in S\cap X_1$. 
Then $y$ must belong to $S\cap X_2$. 
Since $y \in   \overline{S\cap X_1}$ there exists a convergent sequence $\overline{x}=\{x_j\}_{j\in \NN}$ such that 
$y\in L(\overline{x})$.   
Then since $y \not \in X_1$ it follows that $y\in L(\overline{x})$ is not an isolated point. 

The set $S$ is locally closed, hence there are $D\subset X$ open and $F\subset X$ closed such that $S=F\cap D$. 
Since $y\in S\subset D$, we obtain that there is $y'\ne y$ with $y'\in D\cap L(\overline{x})$. 
On the other hand the sequence $\overline{x}$ is contained in $F$, thus $L(\overline{x})\subset F$. 
We get that $\{y, y'\}\in L(\overline{x})\cap S$, which contradicts the fact that $S$ is Hausdorff. 
Therefore we must have that $y\in S\cap X_1$.
\end{proof}

Topologies of the type considered in Lemma~\ref{SQ5} have been earlier studied in connection with $C^*$-algebra extensions.
(See for instance \cite[Ch. VII]{De72}, for a stronger condition on the sets of limit points.)

\begin{definition}\label{SQ6} 
\normalfont
The short exact sequence of separable $C^*$-algebras 
$$
0\to\Ac_1\to\Ac\mathop{\to}\limits^p\Ac_2\to0
$$
 is called a \emph{controlled boundary extension} 
if,  denoting $X=\widehat{A}$, $X_1=\widehat{\Ac_1}$, $X_2=\widehat{\Ac_2}$, the topological space
$X= X_1\sqcup X_2$ satifies the conditions in Lemma~\ref{SQ5}. 
\end{definition}

Now we can prove the following theorem. 

\begin{theorem}\label{SQ7}
Let 
$0\to\Ac_1 \to\Ac\mathop{\to}\limits^p\Ac_2\to0$
be any controlled boundary extension.
 Then the following assertions are equivalent: 
\begin{enumerate}[(i)] 
\item $\Subquot^{\Tr}(\Ac)=\Subquot^{\Tr}_0(\Ac)$
\item $\Subquot^{\Tr}(\Ac_1)=\Subquot^{\Tr}_0(\Ac_1)$ and $\Subquot^{\Tr}(\Ac_2)=\Subquot^{\Tr}_0(\Ac_2)$.
\end{enumerate}
If in addition $\Ac$ is liminary, then $\Subquot^{\Tr}(\Ac)=\Subquot^{\Hausd}(\Ac)$ if and only if 
$\Subquot^{\Tr}(\Ac_1)=\Subquot^{\Hausd}(\Ac_1)$ and $\Subquot^{\Tr}(\Ac_2)=\Subquot^{\Hausd}(\Ac_2)$

\end{theorem}

\begin{proof}
The asssertion (i) $\Rightarrow$ (ii) follows from Lemma~\ref{SQ2}.

For (ii) $\Rightarrow$ (i), let $[\Tc] \in\Subquot^{\Tr} (\Ac)$. 
Then 
$S= \widehat{\Tc}$ is a locally  closed and Hausdorff subset of $X$. 
Since $\Tc$ has continuous trace it follows from \cite[Cor.~10.5.6]{Dix64} that $\Tc$ is a 
$C_0(S)$-algebra. 

On the other hand, by Lemma~\ref{SQ5}, the sets $S\cap X_1$ and $S\cap X_2$ are both open and closed in $S$. 
Thus we can write 
\begin{equation}\label{SQ7-1}
 \Tc = \Tc_1 \dotplus \Tc_2
\end{equation} 
where $\Tc_j$ are ideals in $\Tc$, have continuous trace,  and $\widehat{\Tc_j} = S\cap X_j$, $j=1, 2$. 

We also have that $[\Tc_j]\in\Subquot^{\Tr} (\Ac)$: 
Indeed, since $[\Tc] \in\Subquot^{\Tr} (\Ac)$ there are ideals, $\Ic\subset \Jc\subset \Ac$ with
$\Tc = \Jc/\Ic$. 
Let $p\colon \Jc \to \Jc/\Ic$ be the canonical projection. 
Then $\Ic \subset p^{-1}(\Tc_j)$, $j=1, 2$ are ideals in $\Jc$, hence in $\Ac$, and 
$p^{-1}(\Tc_j)/\Ic\simeq \Tc_j$. 

For $j= 1, 2$, the sets  $S \cap X_j$ are locally closed in $X_j$, 
hence there are $ [\Tc_j ']\in  \Subquot^{\Tr}(\Ac_j)= \Subquot^{\Tr}_0(\Ac_j)$, with $\widehat{\Tc_j '} = S\cap X_j$. 
On the other hand by Lemma~\ref{SQ2}, $[\Tc_j ']\in\Subquot^{\Tr} (\Ac)$, $j=1, 2$.  
By Remark~\ref{S33}, we get that $\Tc_j'\simeq \Tc_j$, hence $[\Tc_j]\in \Subquot_0^{\Tr} (\Ac) $, $j=1, 2$. 
It then follows from \eqref{SQ7-1} that $[\Tc]\in \Subquot_0^{\Tr} (\Ac)$. 

The proof of the second assertion is completely similar, using again  \cite[Cor.~10.5.6]{Dix64}.
\end{proof}

\subsection*{New classes of  nilpotent Lie groups  satisfying Conjecture~\ref{SQ0}}

The assertions of the following lemma, in the special case of two-step nilpotent Lie groups, go back to 
\cite[Lemma~3.3]{LiRo96} and  the proof of \cite[Th.~6.3.3]{Ec96}. 

\begin{lemma}\label{allflat}
Let $G$ be a connected simply connected nilpotent Lie group that has only flat orbits. 
For every integer $d\ge 0$ denote by $(\gg^*/G)_d$ the set of coadjoint orbits of dimension $d$. 
Then we have
\begin{enumerate}[(i)]
\item\label{allflat_item1} The set $(\gg^*/G)_d$ is locally closed and Hausdorff, and the corresponding subquotient in $C^*(G)$ has continuous trace. 
\item\label{allflat_item2} If $S\subset \gg^*/G$ is a locally closed  Hausdorff subset, then $S\cap (\gg^*/G)_d$ is closed in the relative 
topology. 
\end{enumerate}
\end{lemma}
\begin{proof}
The lemma is trivial for $d=0$, therefore let  $d> 0$.

We first prove \eqref{allflat_item1}. 
Let $(\gg^*/G)_{\le d}$ be the set of all orbits of dimension $\le  d$.
It is a closed subset of $\gg^*/G$, and 
$$ (\gg^*/G)_{d}= (\gg^*/G)_{\le d}\setminus (\gg^*/G)_{\le {d-2}}.$$
This shows that $(\gg^*/G)_{d}$ is locally closed.

Let now $\{\Oc_j\}_{j\in \NN}$ be a convergent sequence of coadjoint orbits
in $(\gg^*/G)_d$, and let $\Oc \in L(\{\Oc_j\})\cap (\gg^*/G)_d$. 
Then there exists $\xi_j \in \Oc_j$ and $\xi\in\Oc$ such that 
$
\lim_{j\to \infty} \xi_j = \xi, 
$
and $\Oc_j = \xi_j + \gg(\xi_j)^{\perp}$ and $\Oc = \xi+ \gg(\xi)^{\perp}$. 
We denote $V_j =\gg(\xi_j)^{\perp}$ and $V= \gg(\xi)^{\perp}$, and note that they do not depend on the choice of $\xi_j \in \Oc_j$ and $\xi \in \Oc$, respectively. 
We denote 
$$V_0=\{\theta\in \gg^*\mid (\exists) \theta_j \in V_j, \; \theta_j \to \theta\}.$$
Then it is easy to see that $V_0$ is a subspace of $\gg^*$ of dimension at most $d$.  
Since  $\Oc_j \to \Oc$ we have that $V\subseteq V_0$. 
Moreover $V$ is $d$-dimensional, hence $V=V_0$. 
This shows that $\Oc$ is uniquely determined by the sequence $\{\Oc_j\}$, hence 
the set $(\gg^*/G)_d$ is Hausdorff. 

Moreover, since all the orbits are flat,  $\Oc_j$ converges to $\Oc$ with multiplicity $1$ (see 
\cite[Prop.~2.8]{Lu90}), hence by \cite[Prop.~3.2]{Lu90} we have that 
$$ \lim_{j\to \infty} \Tr \pi_{\Oc_{j}} (f) = \Tr \pi_{\Oc}(f)$$
for every $f$ is a dense $*$-subalgebra of $C^*(G)$. 
Here $\pi_{\Oc_{j}}$, $\pi_{\Oc}$ are irreducible representations associated to $\Oc_j$ and $\Oc$, respectively. 
 This shows  that the subquotient $\Psi^{-1}_{C^*(G)}((\gg^*/G)_d)$ has continuous trace.

To prove \eqref{allflat_item2} let $\{\Oc_j\}_j$ in $S\cap (\gg^*/G)_d$ 
be a sequence of orbits and $\Oc \in S$ with $\Oc_j \to \Oc$.  
We can write $\Oc_j = \xi_j + V_j$ with $\xi_j\in \Oc_j$ as above. 
We regard $\{V_j\}_j$ as a sequence of points in the compact connected manifold $\opn{Gr}(d; \gg^*)$, which is the connected component of the Grassmann manifold consisting of the $d$-dimensional linear subspaces of $\gg^*$. 
Therefore, by selecting a subsequence,  we may assume that there exists $V\in \opn{Gr}(d; \gg^*)$ with $V_j\to V$
in $\opn{Gr}(d; \gg^*)$. 
The set $\xi +V$ is $G$-invariant, and 
$(\xi+V)/G \subseteq L(\{\Oc_j\})$. 
Hence we either have that $\xi+V=\Oc$, in which case $\Oc \in (\gg^*/G)_d$, or $\Oc\in (\xi+V)/G$ is an infinite connected set of lower-dimensional orbits. 
The later alternative is however impossible, 
since $S$ is locally closed, Hausdorff and contains $\Oc_j$. 
\end{proof}

Note that the lemma is not true for general nilpotent Lie groups, as can can be seen from the case of 
threadlike nilpotent Lie groups; see  \cite[Sect.~5]{Lu90}. 
See however \cite[Th.~5.1]{AKLSS01} for the equivalence Fell point $\Longleftrightarrow$  separated point in the set of coadjoint orbits of maximal dimension of the threadlike groups.

The following corollary provides a class of 3-step nilpotent Lie groups for which 
Conjecture~\ref{SQ0} holds true even in a stronger form.

\begin{corollary}\label{SQ8}
If $G$ is any 3-step nilpotent Lie group with 1-dimensional center and with generic flat coadjoint orbits,  
then its $C^*$-algebra $\Ac=C^*(G)$ has the property $\Subquot^{\Hausd}(\Ac)=\Subquot^{\Tr}(\Ac)=\Subquot^{\Tr}_0(\Ac)$. 
\end{corollary}

\begin{proof}
Let $N$ be the center of $G$. 
Then there exists the short exact sequence of $C^*$-algebras
$$ 0\to\Jc\to C^*(G)\to C^*(G/N)\to 0.$$
Moreover, since $G$ has 1-dimensional center and generic flat coadjoint orbits, 
there exists an isometric $*$-isomorphism $\Jc\simeq\Cc_0(\RR\setminus\{0\},\Kc(\Hc))$ 
for some Hilbert space~$\Hc$  
(see for instance \cite[page 126]{Ka95}), 
hence we obtain a short exact sequence 
$$0\to \Cc_0(\RR\setminus\{0\},\Kc(\Hc))\to C^*(G)\to C^*(G/N)\to0$$
that is a controlled boundary extension 
in the sense of Definition~\ref{SQ6}. 
We clearly have $\Subquot^{\Tr}(\Cc_0(\RR\setminus\{0\},\Kc(\Hc)))=\Subquot^{\Tr}_0(\Cc_0(\RR\setminus\{0\},\Kc(\Hc)))$. 
On the other hand, the group $G/N$ is 2-step nilpotent 
(see for instance \cite[Rem. 2.7]{BB13}) 
hence $\Subquot^{\Tr}(C^*(G/N))=\Subquot^{\Tr}_0(C^*(G/N))$ 
by \cite[Th.~3.4]{LiRo96}. 
Therefore by using Theorem~\ref{SQ7} we obtain that $\Subquot^{\Tr}(\Ac)=\Subquot^{\Tr}_0(\Ac)$.

By Lemma~\ref{allflat} it follows that if $T\subset \widehat{G/N}$ is Hausdorff and connected  then 
its points are coadjoint orbits of constant dimension. 
Moreover the set of all coadjoint orbits of fixed dimension of $G/N$ is locally closed subset whose corresponding subquotient of $C^*(G/N)$ has continuous trace. 
Hence by using Theorem~\ref{SQ7} we obtain that $\Subquot^{\Tr}(\Ac)=\Subquot^{\Hausd}(\Ac)$.
\end{proof}

\begin{remark}\label{SQ9}
\normalfont
By using \cite[Ex. 6.8]{BB13} and \cite[Ex. 3.5]{La05}, 
one can construct an uncountable family of mutually nonisomorphic 3-step nilpotent Lie groups $G$ of dimension~7 
with 1-dimensional center $N$ and generic flat coadjoint orbits, 
hence satisfying the hypothesis of Corollary~\ref{SQ8}.
All these groups are (mutually nonisomorphic) extensions of the $2$-step nilpotent Lie group $G_{6, 15}$ 
discussed in \cite[Ex.~6.3.5]{Ec96}, with  the quotients $G/N$ isomorphic with $G_{6, 15}$. 
\end{remark}

In Corollary~\ref{SQ12} below we provide examples of nilpotent Lie groups of arbitrarily high nilpotency step for which Conjecture~\ref{SQ0} holds true. 

\begin{definition}[\cite{HL79}]
\label{SQ10}
\normalfont
For any $m,n\ge1$ we define $\hg_{m,n}$ as the Lie algebra 
with a basis $\{X_1,\dots,X_m\}\cup\{Y_0,\dots,Y_n\}$ and the bracket given by  
$$[X_i,Y_j]=Y_{i+j}$$
for all $i\in\{1,\dots,m\}$ and $j\in\{0,\dots,n\}$ with $i+j\le n$. 
We denote by $X_j^*$, $j=1,\dots, m$, $Y_k^*$, $k=0, \dots, n$ the corresponding dual basis in 
$\hg_{m, n}^*$.  
\end{definition}

\begin{proposition}\label{SQ11}
The following assertions hold: 
\begin{enumerate}[(i)]
\item\label{SQ11_item1} 
For all $m,n\ge 1$ the Lie algebra $\hg_{m,n}$ is $n$-step nilpotent and its 
center $\zg$ is spanned by $\{Y_n\}$ if $m\le n$ and by $\{Y_n\}\cup\{X_{n+1},\dots,X_m\}$ if $m> n$. 
\item\label{SQ11_item2} 
If $n\ge 2$, there exists a Lie algebra isomorphism $\hg_{m,n}/\RR Y_n\simeq\hg_{m,n-1}$. 
\item\label{SQ11_item3} 
If $m\ge n\ge 1$, then the coadjoint isotropy subalgebra at any $\xi\in\hg_{m,n}^*$ with $\langle\xi,Y_n\rangle\ne0$
is $\zg$. 
Hence the corresponding coadjoint orbit is $\Oc_\xi= \xi+\zg^\perp$.
\item\label{SQ11_item4} 
If $m\ge n>k\ge 1$ then the coadjoint isotropy subalgebra at any $\xi\in\hg_{m,n}^*$ with 
$\langle\xi,Y_j\rangle=0$ if $k< j\le n$ and 
$\langle\xi,Y_k\rangle\ne0$ 
is spanned by $\{Y_k,\dots,Y_n\}\cup\{X_{k+1},\dots,X_m\}$. 
\item\label{SQ11_item5} All coadjoint orbits of $H_{m, n}$ are flat. 
\end{enumerate}
\end{proposition}

\begin{proof} 
The Jacobi identity for the Lie bracket of $\hg_{m,n}$ is checked in \cite{Ri87}. 
Assertions~\eqref{SQ11_item1}--\eqref{SQ11_item3} are by-products of the proof of \cite[Lemma 13]{HL79}.  
Then Assertions~\eqref{SQ11_item4} follows by an inductive reasoning using \eqref{SQ11_item1}--\eqref{SQ11_item3}. 
Assertion~\eqref{SQ11_item5} follows by \cite[Lemma 13]{HL79}.
\end{proof}

Note that for $m\ge n$, 
by using the generic flat coadjoint orbits of $H_{m,n}$ given by Proposition~\ref{SQ11}\eqref{SQ11_item3}, 
we obtain 
a short exact sequence 
\begin{equation}\label{SQ12_proof_eq1}
0\to \Cc_0((\RR\setminus\{0\})\times \RR^{m-n} ,\Kc(\Hc))\to C^*(H_{m,n})\to C^*(H_{m,n-1})\to0 
\end{equation}
for some Hilbert space~$\Hc$
(see \cite[page 126]{Ka95}). 
The next lemma shows that the extension \eqref{SQ12_proof_eq1} is a controlled boundary extension.

\begin{lemma}\label{SQ11.5}
Let $m\ge n$. 
Then 
if $\{\xi_j\}_{j\in \NN}$ is a sequence in $\hg^*_{m, n}$ with 
$\langle \xi_j, Y_n\rangle \ne 0$ for all $j\in \NN$, and 
$\langle \xi_j, Y_n\rangle \to 0$, $j\to \infty$, then 
$L(\{\Oc_{\xi_j}\}_{j\in \NN})$ is contained in $\widehat{H_{m, n-1}}$ and has no isolated points.
\end{lemma}

\begin{proof}
Since $\langle \xi_j, Y_n\rangle \ne 0$, by Proposition~\ref{SQ11} we have that 
$$ \Oc_{\xi_j} =\xi_j +\zg^\perp, $$
and the dimension of each $\Oc_{\xi_j}$ is $2n$. 
We write $\xi_j = \xi_j' +\eta_j$, where $\eta_j \in \zg^*/(\RR Y_n^\ast)$, and denote 
$M:= L(\{\eta_j\}_{j\in \NN})\subset \opn{span}\{X^*_{n+1}, \dots, X^*_m\}$.
Then,   
$$ q\colon \hg_{m, n}^* \to  \hg_{m, n}^*/\Ad_{H_{m, n}}^* \simeq \widehat{H_{m, n}}, \quad 
\xi \mapsto q(\xi) = \Oc_\xi,$$ 
then 
$$ q^{-1} (L(\{\Oc_{\xi_j}\}_{j\in \NN}))=\{0\} \times M \times (\hg_{m, n}^*/\zg^*) \simeq 
\{0\} \times M \times \RR^{2n}.$$
On the other hand, for every $\eta\in M$ the set 
$$A_\eta :=\{ \xi\in q^{-1} (L(\{\Oc_{\xi_j}\}_{j\in \NN}))\mid \xi\vert_{\zg} = \eta\}$$
is $\Ad_{H_{m, n}}^*$ invariant, hence it is a reunion of $H_{m, n}$ orbits. 
The set $A_\eta /\Ad_{H_{m, n}}^*$ is connected and  contains more than one point, since otherwise $A_\eta$ would 
be a  $H_{m, n}$ coadjoint orbit  of maximum dimension $2n$ and contained in $Y_n^\perp$, which is imposible
by Proposition~\ref{SQ11}, \eqref{SQ11_item3} and \eqref{SQ11_item4}. 
Therefore we get that 
\begin{equation}\label{SQ11.5-1}
L(\{\Oc_{\xi_j}\}_{j\in \NN})= \bigsqcup\limits_{\eta\in M} A_\eta /\Ad_{H_{m, n}}^*
\end{equation}
has no isolated points. 
\end{proof}

\begin{corollary}\label{SQ12}
If $m\ge n\ge 1$ and $H_{m,n}$ is the connected simply connected Lie group whose Lie algebra is $\hg_{m,n}$, 
then $H_{m,n}$ is $n$-step nilpotent 
and its $C^*$-algebra $\Ac=C^*(H_{m,n})$ has the property 
$\Subquot^{\Hausd}(\Ac)=\Subquot^{\Tr}(\Ac)=\Subquot^{\Tr}_0(\Ac)$. 
\end{corollary}

\begin{proof}
We will prove the assertion by inducion on $n$. 
If $n=1$, then $H_{m,n}$ is abelian, hence the property $\Subquot^{\Hausd}(\Ac)=\Subquot^{\Tr}(\Ac)=\Subquot^{\Tr}_0(\Ac)$ is clear. 

For the induction step  note that we have the short exact sequence
\eqref{SQ12_proof_eq1}
and that Lemma~\ref{SQ11.5} shows it satisfies the conditions in Theorem~6.8. 
Therefore by using Theorem~\ref{SQ7} we obtain $\Subquot^{\Hausd}(C^*(H_{m,n}))=\Subquot^{\Tr}(C^*(H_{m,n}))=\Subquot^{\Tr}_0(C^*(H_{m,n}))$, 
and this concludes the proof. 
\end{proof}

We note here that if $G$ is any nilpotent Lie group for which Conjecture~\ref{SQ0} holds true, 
then the length of the $C^*(G)$, in the sense of Definition~\ref{solv_def}, is less or equal to the cardinal of the coarse stratification of $\gg^*$, which, in general can be much 
smaller than  the cardinal of the fine stratification used in Proposition~\ref{prop3.5}. 
A similar problem for the  (GCT) composition series of $C^*(G)$ has been studied in \cite{Ec96}   
using twisted crossed products.

\appendix

\section{Complements on properly convergent sequences}

Here we discuss some uniqueness properties of the boundary value mappings that occur in Definition~\ref{norcontspec} 
(see Proposition~\ref{PC4} below). 

\begin{definition}\label{PC0}
\normalfont
A sequence $\bar x:=\{x_k\}_{k\in \NN}$ in a topological space $X$ is said to be \emph{properly convergent} 
if its set of cluster points 
$$L(\bar x):=\bigcap_{k\in \NN}\overline{\{x_i\mid i\ge k\}} $$
has the property that for every point $y\in L(\bar x)$  and every subsequence $\{x_{k_j}\}_{j\in\NN}$ 
one has $x_{k_j}\to y$ as $j\to\infty$. 
\end{definition}

\begin{lemma}\label{PC1}
For every $C^*$-algebra $\Ac$ and any closed sets $S_1,S_2\subseteq\widehat{\Ac}$ one has 
$$S_1\subseteq S_2\iff (\forall a\in\Ac) \ \sup_{[\rho]\in S_1}\Vert\rho(a)\Vert\le \sup_{[\rho]\in S_2}\Vert\rho(a)\Vert. $$
\end{lemma}

\begin{proof}
``$\Rightarrow$'' is obvious, 
and the converse follows by \cite[Lemma 2.1]{Fe60}. 
\end{proof}

\begin{lemma}\label{PC2}
Let $\Ac$ be any $C^*$-algebra, $\bar\pi:=\{[\pi_k]\}_{k\in\NN}$ be any properly convergent sequence in $\widehat{\Ac}$. 
Then 
for any closed set $S\subseteq\widehat{\Ac}$ one has 
\begin{eqnarray}
L(\bar\pi)=S
& \iff & 
(\forall a\in\Ac) \ \lim_{k\to\infty}\Vert\pi_k(a)\Vert=\sup_{[\rho]\in S}\Vert\rho(a)\Vert; \nonumber \\
L(\bar\pi)\subseteq S
& \iff & 
(\forall a\in\Ac) \ \lim_{k\to\infty}\Vert\pi_k(a)\Vert\le \sup_{[\rho]\in S}\Vert\rho(a)\Vert; \nonumber \\
L(\bar\pi)\supseteq S
& \iff & 
(\forall a\in\Ac) \ \lim_{k\to\infty}\Vert\pi_k(a)\Vert\ge \sup_{[\rho]\in S}\Vert\rho(a)\Vert. \nonumber 
\end{eqnarray}
\end{lemma}

\begin{proof} 
We prove only the first equivalence in the statement, 
as this proof contains the arguments needed for the second and third equivalences. 

``$\Rightarrow$'' Use \cite[Th. 2.1]{Fe60}.

``$\Leftarrow$'' 
Let $S\subseteq\widehat{\Ac}$ be any closed set with the property 
$$(\forall a\in\Ac) \quad \lim_{k\to\infty}\Vert\pi_k(a)\Vert=\sup_{[\rho]\in S}\Vert\rho(a)\Vert.$$
Since $\bar\pi:=\{[\pi_k]\}_{k\in\NN}$ is a properly convergent sequence, 
it follows by ``$\Rightarrow$'' that 
$$(\forall a\in\Ac) \quad \lim_{k\to\infty}\Vert\pi_k(a)\Vert=\sup_{[\rho]\in L(\bar\pi)}\Vert\rho(a)\Vert$$
hence 
$$(\forall a\in\Ac) \quad \sup_{[\rho]\in S}\Vert\rho(a)\Vert=\sup_{[\rho]\in L(\bar\pi)}\Vert\rho(a)\Vert.$$
Then, as both $L(\bar\pi)$ and $S$ are closed subsets of $\widehat{\Ac}$, 
it follows by Lemma~\ref{PC0} that $L(\bar\pi)\subseteq S$ and $L(\bar\pi)\supseteq S$, 
hence $L(\bar\pi)=S$, and this completes the proof. 
\end{proof}

\begin{lemma}\label{PC3}
Let $\Ac$ be any $C^*$-algebra, and $\bar\pi:=\{[\pi_k]\}_{k\in\NN}$ and $\bar\tau:=\{[\tau_k]\}_{k\in\NN}$ 
be any properly convergent sequences in $\widehat{\Ac}$. 
Then one has 
\begin{equation}\label{PC3_eq1}
L(\bar\pi)\supseteq L(\bar\tau)\iff (\forall a\in\Ac)\ 
\lim_{k\to\infty}\bigl(\Vert\pi_k(a)\Vert-\Vert\tau_k(a)\Vert\bigr)\ge 0
\end{equation}
and 
\begin{equation}\label{PC3_eq2}
L(\bar\pi)= L(\bar\tau)\iff (\forall a\in\Ac)\ 
\lim_{k\to\infty}\bigl\vert\Vert\pi_k(a)\Vert-\Vert\tau_k(a)\Vert\bigr\vert= 0.
\end{equation}
\end{lemma}

\begin{proof}
It follows by Lemma~\ref{PC2} that 
$$(\forall a\in\Ac)\quad \lim_{k\to\infty}\Vert\pi_k(a)\Vert=\sup_{\rho\in L(\bar\pi)}\Vert\rho(a)\Vert$$
and 
$$(\forall a\in\Ac)\quad \lim_{k\to\infty}\Vert\tau_k(a)\Vert=\sup_{\rho\in L(\bar\tau)}\Vert\rho(a)\Vert$$
and then the implications ``$\Rightarrow$'' in both \eqref{PC3_eq1} and \eqref{PC3_eq2} follow at once. 
For \eqref{PC3_eq1}  one also needs the elementary remark that if $\{t_k\}_{k\in\NN}$ and $\{s_k\}_{k\in\NN}$ are any convergent sequences of real numbers, then 
$$\lim_{k\to\infty}t_k=\lim_{k\to\infty}s_k\iff \lim_{k\to\infty}\vert t_k-s_k\vert=0.$$
Furthermore, if the condition in the right-hand side of \eqref{PC3_eq1} is satisfied, 
it follows by the above displayed equalities that 
$$(\forall a\in\Ac)\quad \lim_{k\to\infty}\bigl\vert\Vert\pi_k(a)\Vert 
\ge \sup_{[\rho]\in L(\bar\tau)}\Vert\rho(a)\Vert $$
and then $L(\bar\pi)\supseteq  L(\bar\tau)$ by the last equivalence in Lemma~\ref{PC2}. 
This completes the proof of \eqref{PC3_eq1}. 
The implication ``$\Leftarrow$'' in \eqref{PC3_eq2} can be proved similarly. 
\end{proof}

Now we can establish a kind of uniqueness property of some mappings that occur in Definition~\ref{norcontspec}, 
which also shows how these mappings depend on the limit set of the properly convergent sequence to which they are associated. 
It is worth pointing out that no linearity properties of these mappings are needed in the following proposition.

\begin{proposition}\label{PC4}
In the setting of Definition~\ref{norcontspec}, fix $i\in\{0,\dots,d+1\}$. 
Let $\bar\pi:=\{[\pi_k]\}_{k\in\NN}$ and $\bar\tau:=\{[\tau_k]\}_{k\in\NN}$ 
be any properly convergent sequences contained in $\Gamma_i$ with $L(\bar\pi)\cup L(\bar\tau)\subseteq S_i$. 
Assume that for every $ k\in\NN $  one has some mappings $ \tilde\sigma_{\bar\pi,k}: CB(S_{i})\to \Bc(\H_i)$ 
and $ \tilde\sigma_{\bar\tau,k}: CB(S_{i})\to \Bc(\H_i)$ such that  
$$
\lim_{k\to\infty}\norm{\Fc (a)(\pi_k)-\tilde\sigma_{\bar\pi,k} (\Fc (a)\vert_{S_{i}})}_{\Bc(\Hc_i)}
=\lim_{k\to\infty}\norm{\Fc (a)(\tau_k)-\tilde\sigma_{\bar\tau,k} (\Fc (a)\vert_{S_{i}})}_{\Bc(\Hc_i)}=0
$$
for every $a\in\Ac$. 
Then the following assertions hold: 
\begin{enumerate}
\item If $L(\bar\pi)\supseteq L(\bar\tau)$, then 
$$(\forall a\in\Ac)\quad 
\lim_{k\to\infty}\bigl(\Vert\tilde\sigma_{\bar\pi,k} (\Fc (a)\vert_{S_i})\Vert_{\Bc(\Hc_i)}
-\Vert\tilde\sigma_{\bar\tau,k} (\Fc (a)\vert_{S_i})\Vert_{\Bc(\Hc_i)}\bigr)\ge 0.$$
\item If $L(\bar\pi)= L(\bar\tau)$, then 
$$(\forall a\in\Ac)\quad 
\lim_{k\to\infty}\bigl\vert\Vert\tilde\sigma_{\bar\pi,k} (\Fc (a)\vert_{S_i})\Vert_{\Bc(\Hc_i)}
-\Vert\tilde\sigma_{\bar\tau,k} (\Fc (a)\vert_{S_i})\Vert_{\Bc(\Hc_i)}\bigr\vert= 0.$$
\end{enumerate}
\end{proposition}

\begin{proof}
The hypothesis can be written as 
$$
\lim_{k\to\infty}\norm{\pi_k(a)-\tilde\sigma_{\bar\pi,k} (\Fc (a)\vert_{S_{i}})}_{\Bc(\Hc_i)}
=\lim_{k\to\infty}\norm{\tau_k(a)-\tilde\sigma_{\bar\tau,k} (\Fc (a)\vert_{S_{i}})}_{\Bc(\Hc_i)}=0
$$
and this implies 
$$
\lim_{k\to\infty}\bigl\vert\Vert\pi_k(a)\Vert_{\Bc(\Hc_i)}
-\Vert\tilde\sigma_{\bar\pi,k} (\Fc (a)\vert_{S_{i}})\Vert_{\Bc(\Hc_i)}\bigr\vert
=0
$$
and 
$$\lim_{k\to\infty}\bigl\vert\Vert\tau_k(a)\Vert_{\Bc(\Hc_i)}
-\Vert\tilde\sigma_{\bar\tau,k} (\Fc (a)\vert_{S_{i}})\Vert_{\Bc(\Hc_i)}\bigr\vert=0. $$
Now the assertions follow by Lemma~\ref{PC3}. 
\end{proof}


\begin{thebibliography}{10000000}


\bibitem[AKLSS01]{AKLSS01}
\textsc{R.J.~Archbold, E.~Kaniuth, J.~Ludwig, G.~Schlichting,  D.W.B.~~Somerset},
 Strength of convergence in duals of $C\sp *$-algebras and nilpotent Lie groups. 
 \textit{Adv. Math.} \textbf{158} (2001), no. 1, 26--65.



\bibitem[Ar69]{Ar69} 
\textsc{W.B.~Arveson}, 
Subalgebras of $C^*$-algebras. 
\textit{Acta Math.} \textbf{123} (1969), 141--224.

\bibitem[BB15]{BB13}
\textsc{I.~Belti\c t\u a, D.~Belti\c t\u a}, 
On Kirillov's lemma for nilpotent Lie algebras. 
\textit{J. Algebra} \textbf{427} (2015), no. 1, 85-103.


\bibitem[BL04]{BL04}
\textsc{D.P.~Blecher, Ch.~Le Merdy}, 
\textit{Operator algebras and their modules ---an operator space approach}. 
London Mathematical Society Monographs. New Series, \textbf{30}. 
Oxford Science Publications. The Clarendon Press, Oxford University Press, Oxford, 2004.

\bibitem[Bo07]{Bo07} 
\textsc{N.~Bourbaki},
\textit{Topologie g\'en\'erale}. Chap. 1 \`a 4. 
Springer, 2007.

\bibitem[Bu68]{Bu68}
\textsc{R.C. Busby}, 
Double centralizers and extensions of $C^{\ast}$-algebras. 
\textit{Trans. Amer. Math. Soc.} {\bf 132} (1968), 79--99.

\bibitem[CE76]{CE76}
\textsc{M.D. Choi, E.G.~Effros}, 
The completely positive lifting problem for $C^*$-algebras. 
\textit{Ann. of Math. (2)} \textbf{104} (1976), no.~3, 585--609.

\bibitem[CE77]{CE77}
\textsc{M.D. Choi, E.G.~Effros}, 
Nuclear $C^*$-algebras and injectivity: the general case. 
\textit{Indiana Univ. Math. J.} \textbf{26} (1977), no.~3, 443--446.

\bibitem[CG90]{CG90}
\textsc{L.J.~Corwin, F.P.~Greenleaf}, 
{\it Representations of nilpotent Lie groups and their applications}. Part I. Basic theory and examples. 
Cambridge Studies in Advanced Mathematics \textbf{18}. Cambridge University Press, Cambridge, 1990. 

\bibitem[De72]{De72}
\textsc{C.~Delaroche}, 
Extensions des $C^*$-alg\`ebres. 
\textit{M\'emoires de la Soci\'et\'e Math\'ematique de France} \textbf{29} (1972), 3-142.



\bibitem[Dix61]{Dix61}
\textsc{J.~Dixmier},
Points s\' epar\' es dans le spectre d'une $C^*$-alg\`ebre.
\textit{Acta Sci. Math. Szeged} \textbf{22} (1961), 115--128.

\bibitem[Dix64]{Dix64}
\textsc{J.~Dixmier}, 
\textit{Les $C^{\ast} $-alg\`ebres et leurs repr\'esentations}. 
Cahiers Scientifiques, Fasc. \textbf{XXIX} Gauthier-Villars \& Cie, \'Editeur-Imprimeur, Paris, 1964.

\bibitem[Dy78]{Dy78}
\textsc{A.~Dynin}, 
Inversion problem for singular integral operators: $C^*$-approach. 
\textit{Proc. Nat. Acad. Sci. U.S.A.} \textbf{75} (1978), no.~10, 4668--4670. 

\bibitem[Ec96]{Ec96}
\textsc{S.~Echterhoff}, 
Crossed products with continuous trace. 
\textit{Mem. Amer. Math. Soc.} \textbf{123} (1996), no. 586, viii+134.


\bibitem[Fe60]{Fe60}
\textsc{J.M.G.~Fell},
The dual spaces of $C^*$-algebras.
\textit{Trans. Amer. Math. Soc.} \textbf{94} (1960), 365--403.

\bibitem[Fe62]{Fe62}
\textsc{J.M.G.~Fell},
A Hausdorff topology for the closed subsets of a locally compact non-Hausdorff space. 
\textit{Proc. Amer. Math. Soc.} \textbf{13} (1962), 472--476.




\bibitem[HY88]{HY88}
\textsc{D.~Handelman, H.S.~Yin}, 
Toeplitz algebras and rotational automorphisms associated to polydiscs. 
{\it Amer. J. Math.} {\bf 110} (1988), no. 5, 887--920.


\bibitem[HL79]{HL79}
\textsc{H.J.~Hey, J.~Ludwig}, 
Der Satz von Helson-Reiter f\"ur spezielle nilpotente Lie-Gruppen. 
\textit{Math. Ann.} \textbf{239} (1979), no.~3, 207--218.

\bibitem[La05]{La05}
\textsc{J.~Lauret}, 
Minimal metrics on nilmanifolds. 
In: J. Bure\v s, O. Kowalski, D. Krupka, J. Slov\'ak (eds.), 
\textit{Differential geometry and its applications}. Matfyzpress, Prague, 2005, pp.~79--97.


\bibitem[LiRo96]{LiRo96}
\textsc{R.L.~Lipsman,  J.~Rosenberg},  
The behavior of Fourier transforms for nilpotent Lie groups. 
{\it Trans. Amer. Math. Soc.} {\bf 348} (1996), no.~3, 1031--1050. 


\bibitem[Lu90]{Lu90}
\textsc{J.~Ludwig},  
On the behaviour of sequences in the dual of a nilpotent Lie group. 
\textit{Math. Ann.} \textbf{287} (1990), no. 2, 239--257.


\bibitem[LuRe15]{LuRe15}
\textsc{J.~Ludwig, H.~Regeiba},  
$C^*$-algebras with norm controlled dual limits and nilpotent Lie groups.
\textit{J. Lie Theory} {\bf 25} (2015), no.~3, 613--655. 


\bibitem[LuTu11]{LuTu11}
\textsc{J.~Ludwig,  L.~Turowska}, 
The $C^*$-algebras of the Heisenberg group and of thread-like Lie groups. 
{\it Math. Z.} {\bf 268} (2011), no. 3--4, 897--930. 



\bibitem[Ka95]{Ka95}
\textsc{G.G.~Kasparov}, 
$K$-theory, group $C^*$-algebras, and higher signatures (conspectus). 
In: S.C.~Ferry, A.~Ranicki, J.~Rosenberg (eds.), 
\textit{Novikov conjectures, index theorems and rigidity}, Vol. 1 (Oberwolfach, 1993), London Math. Soc. Lecture Note Ser., 226, Cambridge Univ. Press, Cambridge, 1995, pp.~101--146.




\bibitem[Pe84]{Pe84}
\textsc{N.V.~Pedersen},
On the infinitesimal kernel of irreducible representations of nilpotent Lie groups. 
{\it Bull. Soc. Math. France} {\bf 112} (1984), no.~4, 423--467.

\bibitem[Pe89]{Pe89}
\textsc{N.V.~Pedersen}, 
Geometric quantization and the universal enveloping algebra of a nilpotent Lie group. 
{\it Trans. Amer. Math. Soc.} {\bf 315} (1989), no. 2, 511--563. 

\bibitem[Pe94]{Pe94}
\textsc{N.V.~Pedersen},
Orbits and primitive ideals of solvable Lie algebras. 
{\it Math. Ann.} {\bf 298} (1994), no.~2, 275--326.


\bibitem[Ph87]{Ph87}
\textsc{N.C.~Phillips}, 
\textit{Equivariant $K$-theory and freeness of group actions on $C^*$-algebras}. 
Lecture Notes in Mathematics \textbf{1274},  Springer-Verlag, Berlin, 1987.

\bibitem[RaRo88]{RaRo88}
\textsc{I.~Raeburn, J.~Rosenberg}, 
Crossed products of continuous-trace $C^\ast$-algebras by smooth actions. 
\textit{Trans. Amer. Math. Soc.} \textbf{305} (1988), no. 1, 1--45. 




\bibitem[ReLu14]{ReLu13}
\textsc{H.~Regeiba, J.~Ludwig}, 
The $C^*$-algebra of some $6$-dimensional nilpotent Lie groups.
\textit{Adv.  Pure and Appl. Mathematics.} \textbf{5} (2014), no. 3, 173--200.


\bibitem[Ri87]{Ri87}
\textsc{L.F.~Richardson}, 
$N$-step nilpotent Lie groups with flat Kirillov orbits. 
\textit{Colloq. Math.} \textbf{52} (1987), no. 2, 285--287.

\bibitem[Ro94]{Ro94} 
\textsc{J.~Rosenberg}, 
$C^\ast$-algebras and Mackey's theory of group representations. 
In: R.S.~Doran, \textit{$C^\ast$-algebras: 1943�-1993} (San Antonio, TX, 1993), 
Contemp. Math. \textbf{167}, Amer. Math. Soc., Providence, RI, 1994, pp.~150--181. 



\bibitem[Ta03]{Ta03}
\textsc{M.~Takesaki}, 
{\it Theory of operator algebras}. III. 
Encyclopaedia of Mathematical Sciences \textbf{127}. 
Operator Algebras and Non-commutative Geometry 8. Springer-Verlag, Berlin, 2003.

\bibitem[Vo81]{Vo81}
\textsc{D.~Voiculescu}, 
Remarks on the singular extension in the $C^{\ast} $-algebra of the Heisenberg group. 
\textit{J. Operator Theory} \textbf{5} (1981), no. 2, 147--170.


\end{thebibliography}
\end{document}